\documentclass[11pt]{article}
\usepackage[utf8]{inputenc}
\usepackage[english]{babel}
\usepackage{authblk}
\usepackage{graphicx}        
\usepackage{subfigure}                            
\usepackage{multicol}        
\usepackage[bottom]{footmisc}
\usepackage{placeins}
\usepackage{changes,todonotes}
\usepackage{pgf}
\usepackage{psfrag}
\usepackage{pgfplots}
\pgfplotsset{compat=1.18}

\usepackage{hyperref}
\usepackage{appendix}
\usepackage{amsmath,amssymb,amsfonts}
\usepackage{amsthm}
\usepackage{bm}
\usepackage{mathtools}
\usepackage{color}
\usepackage{stmaryrd}
\usepackage{blindtext}
\usepackage{caption}
\usepackage{multirow}
\usepackage{float}
\usepackage[margin=2cm]{geometry}
\usepackage{mathrsfs}%

\usepackage{xcolor}%
\usepackage{textcomp}%
\usepackage{manyfoot}%
\usepackage{booktabs}%
\usepackage{tikz}

\newtheorem{theorem}{Theorem}
\newtheorem{assumption}{Assumption}%
\newtheorem{lemma}{Lemma}%

\newtheorem{definition}{Definition}%

\newcommand{\llbrace}{\lbrace\hspace{-0.15cm}\lbrace}
\newcommand{\rrbrace}{\rbrace\hspace{-0.15cm}\rbrace}
\newcommand{\norm}[1]{\| #1\|}
\newcommand{\trinorm}[1]{{\vert\kern-0.25ex\vert\kern-0.25ex\vert #1 \vert\kern-0.25ex\vert\kern-0.25ex\vert}}

\newcommand*\circled[1]{\tikz[baseline=(char.base)]{
            \node[shape=circle,draw,inner sep=2pt] (char) {#1};}}

\begin{document}

\title{Polytopal discontinuous Galerkin methods for low-frequency poroelasticity coupled to unsteady Stokes flow}

\author[$\star$]{Michele Botti}
\author[$\star$]{Ivan Fumagalli}
\author[$\star$]{Ilario Mazzieri}

\affil[$\star$]{MOX, Laboratory for Modeling and Scientific Computing, Dipartimento di Matematica, Politecnico di Milano, Piazza Leonardo da Vinci 32, I-20133 Milano, Italy}

\affil[ ]{\texttt {\{michele.botti,ivan.fumagalli,ilario.mazzieri\}@polimi.it}}

\maketitle

	
\begin{abstract}
We focus on the numerical analysis of a polygonal discontinuous Galerkin scheme for the simulation of the exchange of fluid between a deformable saturated poroelastic structure and an adjacent free-flow channel. We specifically address wave phenomena described by the low-frequency Biot model in the poroelastic region and unsteady Stokes flow in the open channel, possibly an isolated cavity or a connected fracture system. The coupling at the interface between the two regions is realized by means of transmission conditions expressing conservation laws. 
The spatial discretization hinges on the weak form of the two-displacement poroelasticity system and a stress formulation of the Stokes equation with weakly imposed symmetry. We present a complete stability analysis for the proposed semi-discrete formulation and derive a-priori hp-error estimates.
\end{abstract}

\maketitle

\section{Introduction}\label{sec1}
This paper presents a new discontinuous Galerkin formulation for the numerical solution of the dynamic Stokes–Biot problem, which models the interaction between the free flow of an incompressible fluid and its interaction with a deformable porous medium. This coupled phenomenon, known as fluid–poroelastic structure interaction, has gained significant attention in recent years due to its wide range of applications. These include geomechanical modeling, hydrogeology, environmental science, and biomedical engineering. Specific examples include predicting and managing gas and oil extraction processes from fractured reservoirs, groundwater flow cleanup in deformable aquifers, industrial filter design, and modeling blood-vessel interactions in blood flow. 

The unsteady Stokes equations govern the fluid dynamics, while the low-frequency poroelasticity system describes the wave propagation in the deformable saturated porous medium. The Stokes and Biot regions are coupled through transmission conditions at the interface that enforce the continuity of normal flux, the Beavers–Joseph–Saffman (BJS) slip condition with friction for the tangential velocity, the stress balance, and the continuity of normal stress.

The Stokes–Biot systems, including Stokes-Darcy flows or fluid-structure interaction problems have been widely studied in the literature, see, e.g. \cite{Discacciati2002, Riviere2005, Ervin2009, Gatica2009, Vassilev2014, Richter2017}. 
The first mathematical analysis of the Stokes–Biot system appeared in \cite{Showalter2005}, where a fully dynamic model was reformulated as a parabolic system to demonstrate well-posedness. A numerical investigation was presented in \cite{badia2009coupling}, where the Navier–Stokes equations were used for free fluid flow, and a variational multiscale finite element method was proposed, offering both monolithic and iterative partitioned solutions. In \cite{Bukac2015c}, a non-iterative operator splitting scheme was introduced for an arterial flow model featuring a thin elastic membrane between two regions, utilizing a pressure-based formulation for flow in the poroelastic domain. The work in \cite{Bukac2015b} considered a mixed Darcy model within the Biot system, employing Nitsche’s method to weakly enforce the continuity of normal flux, while \cite{Ambartsumyan2018} introduced a Lagrange multiplier formulation for imposing this continuity. A dimensionally reduced Brinkman–Biot model for fracture flow in poroelastic media was developed and analyzed in \cite{Bukac2017}. 
Well-posedness for the fully dynamic Navier–Stokes/Biot system, using a pressure-based Darcy formulation, was established in \cite{Cesmelioglu2017}. Coupling the Stokes–Biot system with transport processes was explored in \cite{Ambartsumyan2019b}, and a second-order decoupling scheme for a nonlinear Stokes–Biot model was developed in \cite{Kunwar2020}.
In the recent years, a variety of discretization techniques have been introduced for the Stokes–Biot system, including  mixed finite element methods \cite{Wen2020,LiYotov2022}, a staggered finite element method \cite{bergkamp2020staggered}, and a non-conforming finite element method \cite{Wilfrid2020}.

This paper presents and analyzes a new polygonal discontinuous Galerkin scheme (PolydG) for the unsteady Stokes-Biot system. 
In the poroelastic domain, we address wave propagation using the low-frequency Biot model written in the so-called two-displacement formulation \cite{AntoniettiMazzieriNatipoltri2021}. On the other hand, in the fluid domain, we consider the stress formulation of the Stokes equation, similar to \cite{cancrini2024}, with weakly-imposed symmetry. This choice is suggested by the transmission conditions at the interface between the two domains, expressing conservation laws in terms of relations between stress and flow. 
In particular, this strategy allows to avoid both the introduction of a Lagrange multiplier unknown to enforce the coupling as done, e.g., in \cite{Ambartsumyan2018}, and additional penalty terms at the interface as in \cite{AntoniettiBottiMazzieri2023}.
PolydG discretization have been applied successfully to several studies addressing different problem classes such as: second-order elliptic problems \cite{CangianiDongGeorgoulisHouston_2017} and references therein, parabolic differential equations \cite{Cangiani2017}, flows in fractured porous media \cite{Antonietti2019}, fluid-structure interaction problems \cite{AntoniettiVerani2019}, elastodynamics \cite{AntoniettiMazzieri2018}, nonlinear sound waves \cite{AntoniettiMazzieri2020}, coupled wave propagation problems \cite{Antonietti2020_CMAME,AntoniettiBonaldi2020, AntoniettiMazzieriNatipoltri2021,Antonietti2022_VJM,AntoniettiBottiMazzieri2023}, thermo-elasticity in \cite{Antonietti2023_SISC,bonetti2023numerical}, and multi-physics brain modeling in \cite{Corti2023_M3AS,Corti2023,fumagalli2024polytopal}.

We organize the rest of the paper. In Section \ref{sec:model-problem} we present the mathematical model which includes the derivation of the stress formulation of the Stokes problem and  the continuous weak formulation. The spatial discretization with the PolydG method is addressed in Section \ref{sec:polydg_discr}
together with its stability and error analysis. 
Time integration and numerical experiments are presented in Section \ref{sec:time_int} and Section \ref{sec:nume_res}, respectively. Finally, in Section \ref{sec:conc} we draw some conclusions.

\section{The physical model and governing equations}\label{sec:model-problem}
We introduce the dynamic poroelasticity model, the unsteady Stokes problem, and the transmission conditions describing the interaction between the two systems. Then, we derive the variational formulation of the coupled problem and investigate its well-posedness. 
We start this section by introducing some instrumental notation.

\subsection{Notation}

Let $\Omega \subset \mathbb{R}^d$, $d=2,3$ be an open, convex polygonal domain decomposed as the union of two disjoint, polygonal subdomains, i.e., $\Omega  = \Omega _p\cup \Omega _f$, representing the poroelastic and the fluid domains, respectively. The two subdomains share part of their boundary, resulting in the interface $\Gamma_I = \partial \Omega_p \cap  \partial \Omega_f$. The Lipschitz boundary of $\Omega$  is denoted by $\partial \Omega = \Gamma_{p} \cup \Gamma_{f}$, where $\Gamma_{\diamond} = \partial \Omega_\diamond \setminus \Gamma_I$, for $\diamond = \{p,f\}$,  being  $\partial \Omega_\diamond$ the union of two disjoint portion $\Gamma_\diamond^D$ and $\Gamma_\diamond^N$ with positive measure and where Dirichlet and Neumann conditions are imposed, respectively. 
The outer unit normal vectors to $\partial \Omega_p$ and $\partial \Omega_f$ are denoted by $\bm n_p$ and $\bm n_f$, respectively, so that $\bm n_f$ = -$\bm n_p$ on $\Gamma _I$.

In the following, for a simply connected $X \subset \Omega$ and $\ell \geq 0$, the notation $\bm{H}^\ell(X)$ will be employed in place of $[H^\ell(X)]^d$ for vector valued Sobolev spaces, assuming by convention that $\bm{H}^0(X) \equiv \bm{L}^2(X)$. Moreover, we will denote with $(\cdot, \cdot)_X$ the scalar product in $\bm{L}^2(X)$ and with $||\cdot||_X$ the associated norm.
In addition, we will use $\bm{H}({\rm div},X)$ to denote the space of $\bm{L}^2(X)$ functions with square integrable divergence. A similar notation will be adopted for tensor-valued functions, i.e. $\mathbb{L}^2(X)$ and $\mathbb{H}({\rm div},X)$ stand for $[\bm{L}^2(X)]^d$ and $[\bm{H}({\rm div},X)]^d$, respectively.
For a given final time $T>0$, $k\in\mathbb{N}$, and a Hilbert space $H$, the usual notation $C^k([0,T];H)$ is adopted for the space of $H$-valued functions, $k$-times continuously differentiable in $[0,T]$. 
The notation $x\lesssim y$ stands for $x\leq C y$, with $C>0$, independent of the discretization parameters, but possibly dependent on the physical coefficients and the final time $T$.

In the following, for tensor fields $\bm \tau$ we will use the notation 
\begin{equation*}
{\rm tr}(\bm{\tau}) = \sum_{i=1}^d \tau _{ii}, \quad {\rm dev}(\bm{\tau}) = \bm{\tau} - \frac{1}{d}{\rm tr}(\bm{\tau})\bm{I}, \quad  {\rm skew}(\bm \tau) =  
\frac{\bm{\tau} -\bm{\tau}^{\rm T}}2 
\end{equation*}
to indicate the trace, the deviatoric part, and the skew symmetric part of $\bm \tau$, respectively. 

\subsection{The poroelasto-fluid problem}\label{sec:poroelastic-fluid-model}

In the poroelastic domain $\Omega_p$, for a final observation time $T > 0$, we consider the following Biot equations: 
\begin{equation}\label{eq:biot-system}
   \begin{cases} 
    \rho \ddot{\bm{u}}_p + \rho _f \ddot{\bm{w}}_p - 
    \nabla \cdot \bm{\sigma}_p  = \bm{f}_p, 
     & in  \; \Omega _p \times (0,T], \\
    \rho_f \ddot{\bm{u}}_p + \rho_w \ddot{\bm{w}}_p + \frac{\eta}{\kappa} \dot{\bm{w}}_p + \nabla p_p = \bm{g}_p, 
 & in \;  \Omega _p \times (0,T],\\
  \bm u_p = \bm 0, & on \;  \Gamma_p^D \times (0,T], \\
  \bm w_p \cdot \bm n_p = 0, & on \;  \Gamma_p^D \times (0,T], \\
  \bm \sigma_p \bm n_p = \bm g_p^N, & on \;  \Gamma_p^N \times (0,T], \\
   p_p = q_p^N, & on \;  \Gamma_p^N \times (0,T], \\
  \bm u_p = \bm u_{p0}, \quad \dot{\bm u}_p = \bm v_{p0}, & in \;  \Omega_p \times \{0\}, \\
  \bm w_p = \bm w_{p0}, \; \; \dot{\bm w}_p = \bm z_{p0}, & in \;  \Omega_p \times \{0\},  
    \end{cases}
\end{equation}
 where $\bm{u}_p$ is the solid and $\bm w_p $ is the filtration displacement, respectively.
In \eqref{eq:biot-system}, the average density $\rho$  is given by $\rho  = \phi \rho _f +(1 - \phi )\rho _s$, where $\rho _s > 0 $ is the solid density, $\rho _f > 0$ is the saturating fluid density, and $\rho _w$ is defined as $\rho _w = \frac{a}{\phi} \rho _f$, with $\phi$ being the porosity satisfying $0 < \phi _0 \leq  \phi  \leq  \phi _1 < 1$ and $a \geq 1$ the tortuosity measuring the deviation of the fluid paths from straight streamlines.
In \eqref{eq:biot-system}, $\eta  > 0$ represents the dynamic viscosity of the fluid, $k > 0$ is the absolute permeability, $\bm{f}_p, \bm{g}_p, \bm{g}_p^D$ and $q_p^N$ are given (regular enough) loading and source terms, respectively, and $\bm u_{p0}, \bm v_{p0}, \bm w_{p0}$, and $\bm z_{p0}$ are regular enough given initial conditions.
In $\Omega _p$, we assume the following constitutive laws which allow to express the pore pressure $p_p$ and stress tensor $\bm{\sigma}_p$ in terms of the two displacements $\bm{u}$ and $\bm{w}$:
\begin{equation}\label{eq:constitutive-poro}
p_p(\bm{u},\bm{w}) = -m (\beta \nabla \cdot \bm{u} + \nabla \cdot \bm{w}),  \qquad 
\bm{\sigma}_p  (\bm{u},\bm{w}) = \bm \sigma_e(\bm u)  - \beta p_p(\bm{u},\bm{w}) \bm{I},  
\end{equation}
where the elastic stress $\bm \sigma_e(\bm u) = \mathbb{C} \bm \varepsilon(\bm u)  = 2\mu \bm{\varepsilon} (\bm{u}) + \lambda (\nabla \cdot \bm{u}) \bm{I}
$, being $\mathbb{C}$ the stiffness tensor and $\bm \varepsilon (\bm{u}) = \frac{1}{2} (\nabla \bm{u} + {\nabla \bm{u}}^T)$ the strain tensor (symmetric gradient) . In \eqref{eq:constitutive-poro}, $\lambda  \geq  0$ and $\mu  \geq  \mu_0 > 0$ are the Lam\'e coefficients of the elastic skeleton. The Biot--Willis coefficient $\beta$  and Biot modulus $m$ are such that $\phi  < \beta  \leq  1$ and $m \geq  m_0 > 0$. 

In the fluid domain $\Omega _f$, we consider a free incompressible viscous fluid with mass density $\rho _f > 0$ and dynamic viscosity $\mu _f > 0$. Assuming that the fluid viscosity is sufficiently high, 
the Stokes' system of equations governs the fluid flow: 
\begin{equation}\label{eq:stokes_up}
    \begin{cases}
         \dot{\bm{u}}_f -  \rho_f^{-1}\nabla \cdot \bm \sigma_f
         = \bm{h}_f,  
        &  in  \;  \Omega _f \times (0,T], \\
        \nabla \cdot \bm{u}_f = 0, & 
         in \; \Omega_f \times (0,T], \\
        \bm{u}_f = \bm{g}_f^D, &
        on \; \Gamma^D_f, \times (0,T],\\
        \bm{\sigma}_f \bm n_f = \bm{g}_f^N, &
        on \; \Gamma^N_f \times (0,T],\\
        \bm{u}_f = \bm{u}_{f0}, &
        in \; \Omega_f \times \{0\} ,      
    \end{cases}
\end{equation}
where $\bm{u}_f$ and $p_f$ are the fluid velocity and pressure, respectively, $\bm \sigma_f  = 2 \mu _f \bm   \varepsilon (\bm{u}_f) - p_f \bm I$ is the fluid stress tensor and $\bm{h}_f$ is (a regular enough) body force per unit mass exerted on the fluid. In \eqref{eq:stokes_up}, $\bm g_f^D, \bm g_f^N$ and $\bm u_{f0}$ are regular enough boundary and initial conditions, respectively.  
The first equation in \eqref{eq:stokes_up} represents the conservation of total momentum of the flow, while the second the mass conservation.
In this paper, we consider a different formulation of the Stokes problem, which is more convenient to accurately represent the momentum conservation and formulate the coupling conditions that will be discussed later.
To this aim, we define 
\begin{equation}\label{eq:def-sigma}
    \bm{\Sigma}_f(t) =  \int_{0}^{t} 
    \bm \sigma_f(s) \, ds,
\end{equation}
and integrate in time over $(0,t)$ the first equation in \eqref{eq:stokes_up} to get 
\begin{equation}\label{eq:1st-eq-rewrite}
 \bm{u}_f(t)  - \rho_f^{-1} \nabla \cdot \bm{\Sigma}_f(t) =  \int_{0}^t \bm{h}_f(s) \,ds + \bm{u}_{f0}.
\end{equation}
Using now  definition \eqref{eq:def-sigma} we can infer that
\begin{equation}\label{eq:dev}
    \frac{1}{2\mu_f} {\rm dev}(\dot{\bm \Sigma}_f) = \bm\varepsilon (\bm{u}_f),
\end{equation}
which directly encodes the incompressibility constraint (second equation) in \eqref{eq:stokes_up}.
Next, by introducing the rotation $\bm  r_f =  \nabla \bm{u}_f - \bm{\varepsilon} (\bm{u}_f) $ and combining \eqref{eq:1st-eq-rewrite} and \eqref{eq:dev} we get 
\begin{equation}\label{eq:stokes-sigma}
    \begin{cases}
        (2\mu_f)^{-1} {\rm dev}(\dot{\bm \Sigma}_f)  -  \nabla (\rho _f^{-1} \nabla \cdot \bm{\Sigma}_f) +\bm r_f = \bm F_f ,
        & in  \; \Omega_f \times (0,T],\\ 
        {\rm skew} (\dot{\bm \Sigma}_f) =  \bm 0, & in  \; \Omega_f \times (0,T],\\
       \rho _f^{-1}\nabla \cdot \bm{\Sigma}_f = \bm G_f,  &   on  \; \Gamma_{f}^D \times (0,T],  \\
       \bm{\Sigma}_f \bm n_f = \int_0^t \bm g_f^N(s) \, ds,&
        on \; \Gamma^N_f \times (0,T],\\
        \bm \Sigma_f = \bm 0, &   in  \; \Omega_{f} \times \{0\},
    \end{cases}
\end{equation}
$$
\text{with }
\bm F_f = \nabla\left(\int_{0}^t \bm{h}_f(s) \,ds + \bm{u}_{f0}\right) \;\text{ and }\; 
\bm G_f = \bm g_f^D -\left(\int_{0}^t \bm{h}_f(s) \,ds - \bm{u}_{f0}\right)_{|\partial\Omega_f}.
$$

The poroelastic-fluid coupling is achieved by imposing interface conditions that must account for the conservation of mass and overall momentum. Therefore, these conditions will encompass the continuity of both the normal fluid flux and the stress. Two additional constitutive relations are involved: one describes the relationship between the filtration velocity and the pressure increment, the other addresses the impact of the tangential stress component on the velocity increment. The former is the Robin boundary condition, while the  latter is the Beavers-Joseph-Saffman (BJS) slip condition.
Hence,  on $ \Gamma _I \times (0,T]$ we impose: 
\begin{equation}\label{eq:coupling-conditions}
    \begin{cases}
        (\alpha \dot{\bm{u}}_p + \dot{\bm{w}}_p) \cdot \bm{n}_p = \bm{u}_f \cdot \bm{n}_p, & ({\rm flux \, conservation}), \\
        \dot{\bm{\Sigma}}_f \bm{n}_p \cdot \bm{n}_p = 
        \gamma \dot{\bm{w}}_p \cdot \bm{n}_p - p_p, & ({\rm Robin \, condition}), \\
        \alpha \dot{\bm{\Sigma}}_f \bm{n}_p \cdot \bm{n}_p
        = \bm{\sigma}_p \bm{n}_p \cdot \bm{n}_p, & ({\rm normal \, stress \, conservation}), \\
        \dot{\bm{\Sigma}}_f \bm{n}_p \wedge \bm{n}_p = \bm{\sigma}_p \bm{n}_p \wedge \bm{n}_p = 
        \delta (\bm{u}_f - \dot{\bm{u}}_p) \wedge \bm{n}_p, & ({\rm BJS \, condition}).
    \end{cases}
\end{equation}
Here $\alpha > 0$ is a coefficient related to the fraction of the contact surface $\Gamma _I$ where the diffusion paths of the porous medium are exposed to the fluid in the open channel, 
$\gamma \geq 0$ is the fluid entry resistance and 
$\delta >0$ depends on the slip rate coefficient and the conductivity tensor.
Furthermore, for a vector field $\bm v$ defined on $\Gamma_I$,  the notation $\bm{v}\wedge \bm n_p$ denotes the tangential component of $\bm v$. In the case $d=2$, we have $\bm v\wedge \bm n_p=\bm v\cdot \bm t_p$, where $\bm{t}_p$ is the tangential unit vector to the interface $\Gamma _I$, directed in such a way that the angle measured from $\bm{t}_p$ to $\bm{n}_p$ is positive. 

Finally, the poroelastic-fluid interaction problem
is obtained by combining \eqref{eq:biot-system} and \eqref{eq:stokes-sigma} with conditions \eqref{eq:coupling-conditions}.

\subsection{Weak Formulation}
For the sake of presentation, in the following, we will consider $\bm g_p^N = \bm 0$, $q_p^N = 0$ on $\Gamma_p^N$  in \eqref{eq:biot-system} and $\bm g_f^D = \bm g_f^N = \bm 0$ on $ \Gamma_f^N$ in  \eqref{eq:stokes-sigma}. The general case can be treated analogously. 
We introduce the Sobolev space 
$$
\bm V = \bm H^1_{0,\Gamma_p^D}(\Omega_p) \times \bm H_{0,\Gamma_p^D}({\rm div},\Omega_p) \times \mathbb{H}_{0,\Gamma_f^N}({\rm div},\Omega_f) \times [L^2(\Omega_f)]^{d^*},
$$
with $d^* = 1$ for $d=2$ and $d^* = 3$ for $d=3$ corresponding to the dimension of skew-symmetric matrices in $\mathbb{R}^{d\times d}$.
Then, the weak formulation of Biot-Stokes reads as follows: $\forall t \in (0,T]$, find $(\bm u_p, \bm w_p, \bm{\Sigma}_f, \bm r_f)(t) \in \bm V$ such that $\forall (\bm{v},\bm{z},\bm{\tau}, \bm \lambda)\in \bm V$ it holds
\begin{multline}\label{eq:weak-form}
\mathcal{M}^p((\ddot{\bm u}_p,\ddot{\bm{w}}_p),(\bm v,\bm z)) + 
\mathcal{M}^f(\dot{\bm{\Sigma}_f},\bm{\tau})
+ \mathcal{D}^p(\dot{\bm{w}}_p, \bm z) +  \mathcal{D}^f(\dot{\bm{\Sigma}}_f, \bm \tau) \\ + \mathcal{A}^p((\bm{u}_p,\bm{w}_p),(\bm{v},\bm{z})) 
+ \mathcal{A}^f(\bm{\Sigma}_f, \bm{\tau}) + \mathcal{B}^f(\bm r_f, \bm \tau) - \mathcal{B}^f(\bm \lambda, \dot{\bm \Sigma}_f) \\
+ \mathcal{C}^{pf} ((\dot{\bm{u}}_p,\dot{\bm{w}}_p),\bm{\tau})  - \mathcal{C}^{fp} (\dot{\bm{\Sigma}_f},(\bm{v},\bm{z})) = \mathcal{F}(\bm{v},\bm{z},\bm{\tau}) 
\end{multline}
with initial conditions given as in \eqref{eq:biot-system} and $\eqref{eq:stokes-sigma}$ and where for any $\bm u, \bm v \in  \bm H^1_{0,\Gamma_p^D}(\Omega_p)  $, $\bm w, \bm z \in \bm H_{0,\Gamma_p^D}({\rm div},\Omega_p) $, $\bm \Sigma, \bm \tau \in \mathbb{H}_{0,\Gamma_f^N}({\rm div},\Omega_f)$ and $\bm r, \bm \lambda \in [L^2(\Omega_f)]^{d^*}$ we have 
\begin{align}\label{eq:bilinear_forms}
\mathcal{M}^p((\bm u,\bm w),(\bm v,\bm z)) & =(\rho \bm u,\bm v)_{\Omega_p} + (\rho_f \bm w,\bm v)_{\Omega_p} + (\rho_f \bm u,\bm z)_{\Omega_p} + (\rho_w \bm w,\bm z)_{\Omega_p}, \nonumber \\[1mm]
\mathcal{M}^f(\bm \Sigma, \bm \tau) & = ((2\mu_f)^{-1} {\rm{dev}}(\bm{\Sigma}),{\rm{dev}}(\bm\tau))_{\Omega_f},  \nonumber \\[1mm]
\mathcal{D}^p(\bm w, \bm z) & = (\eta k^{-1}\bm w ,\bm z )_{\Omega_p} +  \langle\gamma \bm{w} \cdot \bm{n}_p, \bm{z} \cdot \bm{n}_p\rangle_{\Gamma_I},  \nonumber \\[1mm]
\mathcal{D}^f(\bm \Sigma, \bm \tau) & = (\delta^{-1} \bm{\Sigma} \bm{n}_p \wedge \bm{n}_p, \bm{\tau} \bm{n}_p \wedge \bm{n}_p )_{\Gamma_I},  \nonumber \\[1mm]
\mathcal{A}^p((\bm{u},\bm{w}),(\bm{v},\bm{z})) & = (\bm{\sigma}_e(\bm{u}), \bm{\varepsilon}(\bm{v}))_{\Omega_p}  
+ (m(\beta \nabla \cdot \bm{u} + \nabla \cdot \bm{w}), \beta \nabla \cdot \bm{v} + \nabla \cdot \bm{z})_{\Omega_p},   \\[1mm]
\mathcal{A}^f(\bm{\Sigma},\bm{\tau}) & = (\rho_f^{-1} \nabla \cdot \bm{\Sigma},\nabla \cdot \bm{\tau})_{\Omega_f},  \nonumber
\\[1mm]
\mathcal{B}^f(\bm r, \bm \tau) & = (\bm r, {\rm skew}(\bm \tau))_{\Omega_f},  \nonumber
\\[1mm]
C^{pf}((\bm{u},\bm{w}),\bm{\tau})  & = \langle(\alpha \bm{u}+\bm{w}) \cdot \bm{n}_p, \bm{\tau} \bm{n}_p \cdot \bm{n}_p\rangle_{\Gamma_I} + \langle\bm{u} \wedge \bm{n}_p, \bm{\tau} \bm{n}_p \wedge \bm{n}_p\rangle_{\Gamma_I},  \nonumber \\[1mm]
C^{fp}(\bm{\Sigma},(\bm{v},\bm{z}))  & = \langle\bm{\Sigma} \bm{n}_p \cdot \bm{n}_p, (\alpha \bm{v} + \bm{z}) \cdot \bm{n}_p\rangle_{\Gamma_I} + \langle\bm{\Sigma} \bm{n}_p \wedge \bm{n}_p, \bm{v} \wedge \bm{n}_p\rangle_{\Gamma_I},  \nonumber\\[1mm]
\mathcal{F}(\bm{v},\bm{z},\bm{\tau})  & = 
(\bm{f}_p,\bm{v})_{\Omega_p}+(\bm{g}_p,\bm{z})_{\Omega_p}+
(\bm F_f,\bm{\tau})_{\Omega_f} - \langle \bm G_f, \bm \tau \bm n_f \rangle_{\Gamma_I\cup\Gamma_f^D},  \nonumber
\end{align}
where $\langle\cdot,\cdot\rangle_{\Gamma_I}$ denotes the $H^{\frac{1}{2}}(\Gamma_I)$ - $H^{-\frac{1}{2}}(\Gamma_I)$ duality product. We also remark that, according to the assumption $\bm g_f^D = \bm 0$ and integration by parts, we can rewrite the third and fourth term appearing in the definition of the linear functional $\mathcal{F}$ to obtain the alternative expression
$$
\mathcal{F}(\bm{v},\bm{z},\bm{\tau}) = 
(\bm{f}_p,\bm{v})_{\Omega_p}+(\bm{g}_p,\bm{z})_{\Omega_p}-
\left(\int_0^t\bm{h}_f(s)\ ds + \bm u_{f0},\nabla\cdot\bm{\tau}\right)_{\Omega_f}.
$$


\subsection{Stability analysis}

This section presents the stability analysis for the continuous problem \eqref{eq:weak-form}. The arguments in the proof of the main theorem will also be used in the discrete setting in Section \ref{sec:polydg_discr}. For any $(\bm u, \bm w, \bm \Sigma) \in C^1([0,T];\bm H^1_{0,\Gamma_p^D}(\Omega_p) \times \bm H_{0,\Gamma_p^D}({\rm div},\Omega_p) \times \mathbb{H}_{0,\Gamma_f^N}({\rm div},\Omega_f))$ we introduce the energy norms
\begin{align*}
\| (\bm u, \bm w, \bm \Sigma) (t)\|^2_{\mathcal{E}}  &=    \| (\bm u, \bm w)(t) \|^2_{\mathcal{E}_p} + \| \bm \Sigma (t) \|^2_{\mathcal{E}_f},\\
    \| (\bm u, \bm w)(t) \|^2_{\mathcal{E}_p} & =  \|  \dot{\bm u}(t) \|^2_{\Omega_p} +  \|  \dot{\bm w}(t) \|^2_{\Omega_p} +  \| (\eta/ k)^{\frac12} {\bm{w}} (t) \|^2_{\Omega_p}   \\ & \quad + \|\gamma^\frac12 {\bm{w}} \cdot \bm{n}_p(t)\|_{\Gamma_I}^2
  +  \| \mathbb{C}^\frac12 \bm \varepsilon(\bm u) (t) \|_{\Omega_p}^2 + \| m^\frac12 \nabla \cdot  (\beta \bm u + \bm w)(t) \|_{\Omega_p}^2,  \\
\| \bm \Sigma (t) \|^2_{\mathcal{E}_f} & = 
  \| (2\mu_f)^{-\frac12} {\rm dev}({\bm{\Sigma}}) (t)\|^2_{\Omega_f} 
 +  \|\rho_f^{-\frac12} \nabla \cdot \bm \Sigma (t)\|^2_{\Omega_f}  +  \| \delta^{-\frac12} {\bm{\Sigma}} \bm{n}_p \wedge \bm{n}_p (t)\|_{\Gamma_I}^2.
\end{align*}

\begin{lemma}\label{lemma:stab-poro}
The bilinear forms $\mathcal{M}^p$, $\mathcal{A}^p$, and $\mathcal{D}^p$ defined in \eqref{eq:bilinear_forms} are such that for any $\bm u, \bm v \in \bm H^1_{0,\Gamma_p^D}(\Omega_p)$ and any $ \bm w, \bm z \in \bm H_{0,\Gamma_p^D}({\rm div},\Omega_p)$ it holds
\begin{align}
\mathcal{M}^p((\bm{u},\bm{w}),(\bm{v},\bm{z})) & \lesssim (\| \bm{u} \|_{\Omega_p} + \| \bm{w} \|_{\Omega_p})( \| \bm{v} \|_{\Omega_p} + \| \bm{z} \|_{\Omega_p}), \label{eq:M-cont}\\ 
\mathcal{M}^p((\bm{u},\bm{w}),(\bm{u},\bm{w})) & \gtrsim  \|\bm{u}\|_{\Omega_p}^2 + \|\bm{w}\|_{\Omega_p}^2,  \label{eq:M-coer} \\
\mathcal{A}^p((\bm{u},\bm{w}),(\bm{v},\bm{z})) + \mathcal{D}^p(\bm{w},\bm{z}) & \lesssim 
\|\bm u\|_{1,\Omega_p} \| \bm v \|_{1,\Omega_p} + \|\bm w\|_{{\rm div},\Omega_p}\|\bm z\|_{{\rm div},\Omega_p}, \label{eq:A-cont}\\
\mathcal{A}^p((\bm{u},\bm{w}),(\bm{u},\bm{w})) + \mathcal{D}^p(\bm{w},\bm{w}) & \gtrsim 
\| \bm u \|_{1,\Omega_p}^2 + \| \bm w\|_{{\rm div}, \Omega_p}^2. \label{eq:A-coer}
\end{align}
\end{lemma}
\begin{proof}
    See \cite[Lemma 2.3]{AntoniettiMazzieriNatipoltri2021}. 
\end{proof}

\begin{lemma}\label{lemma:stab-fluid}
The bilinear forms $\mathcal{M}^f$ and $\mathcal{A}^f$ defined in \eqref{eq:bilinear_forms} are such that for any $\bm \Sigma, \bm \tau \in \mathbb{H}_{0,\Gamma_f^N}({\rm div},\Omega_f)$
\begin{align}
\mathcal{M}^f(\bm{\Sigma},\bm{\tau}) + \mathcal{A}^f(\bm{\Sigma},\bm{\tau}) & \lesssim \| \bm \Sigma \|_{{\rm div},\Omega_f}\| \bm \tau \|_{{\rm div},\Omega_f}, \label{eq:AM-cont-fluid}\\ 
\mathcal{M}^f(\bm{\Sigma},\bm{\Sigma}) + \mathcal{A}^f(\bm{\Sigma},\bm{\Sigma})  & \gtrsim 
 \| \bm \Sigma \|_{{\rm div},\Omega_f}^2, \label{eq:AM-coer-fluid}.
\end{align}
\end{lemma}
\begin{proof}
The continuity in \eqref{eq:AM-cont-fluid} directly follows from the definition of the deviatoric operator, whereas for \eqref{eq:AM-coer-fluid} we refer the reader to \cite[Lemma 2.2]{cancrini2024}. 
\end{proof}

\begin{theorem}\label{teo:stability}
For any time $t\in (0,T]$ let $(\bm u_p, \bm w_p, \bm \Sigma_f, \bm r_f)(t) \in \bm V$ be the solution to problem \eqref{eq:weak-form}. Then, it holds

\begin{equation*}
\sup_{t\in(0,T]}\| (\bm u_p, \bm w_p, \bm \Sigma_f) (t)\|_{\mathcal{E}}
 \lesssim  \mathcal{G}_0 + 
 \int_0^T \mathcal{N}(\bm f_p, \bm g_p, \bm F_f, \bm G_f)(s) \, ds,
\end{equation*}
where
\begin{align}
    &\mathcal{G}_0^2  = 
 \| (\bm u_p, \bm w_p, \bm \Sigma_f) (0)\|_{\mathcal{E}}^2 
 + \sup_{t\in[0,T]} \left(\|\bm F_f(t)\|^2_{\Omega_f} +  \|\bm G_f(t)\|_{\Gamma_I\cup\Gamma_f^D}^2\right), \label{eq:G_def}\\
    &\mathcal{N}(\bm f_p, \bm g_p, \bm F_f, \bm G_f)(t)  = \| \bm f_p(t) \|_{\Omega_p}
+ \| \bm g_p(t) \|_{\Omega_p} + \|\dot{\bm H}_f(t)\|_{\Omega_f} + \|\dot{\bm G}_f(t)\|_{\Gamma_I\cup\Gamma_f^D}. \label{eq:N_def}
\end{align}
\end{theorem}
\begin{proof}

We consider $(\bm{v},\bm{z},\bm{\tau}, \bm \lambda) = (\dot{\bm{u}_p},\dot{\bm{w}}_p,\dot{\bm{\Sigma}}_f, \bm r_f) $ in \eqref{eq:weak-form} to get 
\begin{multline*}
\mathcal{M}^p((\ddot{\bm u}_p,\ddot{\bm{w}}_p),(\dot{\bm u}_p,\dot{\bm w}_p)) + 
\mathcal{M}^f(\dot{\bm{\Sigma}}_f,\dot{\bm{\Sigma}}_f)
+ \mathcal{D}^p(\dot{\bm{w}}_p, \dot{\bm w}_p) +  \mathcal{D}^f(\dot{\bm{\Sigma}}_f, \dot{\bm{\Sigma}}_f) \\ + \mathcal{A}^p((\bm{u}_p,\bm{w}_p),(\dot{\bm{u}}_p,\dot{\bm{w}}_p)) 
+ \mathcal{A}^f(\bm{\Sigma}_f, \dot{\bm{\Sigma}}_f)
= \mathcal{F}(\dot{\bm{u}}_p,\dot{\bm{w}}_p,\dot{\bm{\Sigma}}_f). 
\end{multline*}
Integrating in time between $0$ and $t$ the above equation leads to 
\begin{multline*}
\frac12 \mathcal{M}^p((\dot{\bm u}_p,\dot{\bm{w}}_p),(\dot{\bm u}_p,\dot{\bm w}_p))(t) + \int_0^t\mathcal{D}^p(\dot{\bm{w}}_p, \dot{\bm w}_p)(s) \, ds   
+ \frac12 \mathcal{A}^p((\bm{u}_p,\bm{w}_p),(\bm{u}_p,\bm{w}_p))(t)      \\
+ \int_0^t \mathcal{M}^f(\dot{\bm{\Sigma}}_f,\dot{\bm{\Sigma}}_f)(s) \, ds  
 + \frac12 \mathcal{A}^f({\bm{\Sigma}}_f,{\bm{\Sigma}}_f)(t) + \int_0^t \mathcal{D}^f(\dot{\bm{\Sigma}}_f,\dot{\bm{\Sigma}}_f)(s)  \,ds  \\ 
   = \int_0^t \mathcal{F}(\dot{\bm{u}}_p,\dot{\bm{w}}_p,\dot{\bm{\Sigma}}_f)(s) \, ds 
+\frac12 \mathcal{M}^p((\bm v_{p0},\bm{z}_{p0}),(\bm v_{p0},\bm{z}_{p0}))  
\\ + \frac12 \mathcal{A}^p((\bm{u}_{p0},\bm{w}_{p0}),(\bm{u}_{p0},\bm{w}_{p0}))  + \frac12 \mathcal{A}^f({\bm{\Sigma}}_{f0},{\bm{\Sigma}}_{f0}).
\end{multline*}
Next, using that $\mathcal{B}(\bm \psi, \bm \psi)(t) \lesssim \mathcal{B}(\bm \psi, \bm \psi)(0) + \int_0^t \mathcal{B}(\dot{\bm \psi}, \dot{\bm \psi})(s)\, ds$ for $\mathcal{B}(\cdot,\cdot) = \{ \mathcal{D}^p(\cdot,\cdot), \mathcal{D}^f(\cdot,\cdot), \mathcal{M}^f(\cdot,\cdot)\}$  with $\bm \psi = \{ \bm w, \bm \Sigma, \dot{\bm \Sigma}\}$, respectively, together with \eqref{eq:M-cont}-\eqref{eq:M-coer}, \eqref{eq:A-cont}-\eqref{eq:A-coer} and \eqref{eq:AM-cont-fluid}-\eqref{eq:AM-coer-fluid}, we obtain
\begin{equation}\label{eq:partial-stability-estimate}
 \| (\bm u, \bm w)(t) \|^2_{\mathcal{E}_p} + \| \bm \Sigma (t) \|^2_{\mathcal{E}_f}  
   \lesssim \int_0^t \mathcal{F}(\dot{\bm{u}}_p,\dot{\bm{w}}_p,\dot{\bm{\Sigma}}_f)(s) \, ds +
\| (\bm u_{p0}, \bm w_{p0}) \|^2_{\mathcal{E}_p} + \| \bm \Sigma_{f0} \|^2_{\mathcal{E}_f}.
\end{equation}
Then, we apply the Cauchy-Schwarz inequality for the forcing terms in the porous domain, while we integrate by parts the other terms to get  
\begin{multline*}
\int_0^t \mathcal{F}(\dot{\bm{u}}_p,\dot{\bm{w}}_p,\dot{\bm{\Sigma}}_f)(s) \, ds \lesssim \int_0^t (\| \bm f_p(s) \|_{\Omega_p}
+ \| \bm g_p(s) \|_{\Omega_p})
\|(\bm{u}_p,\bm{w}_p) (s) \|_{\mathcal{E}_p}\, ds \\
+ (\bm F_f, \bm\Sigma_f)_{\Omega_f}(t) - 
(\bm F_f, \bm\Sigma_f)_{\Omega_f}(0) - \int_0^t (\dot{\bm F}_f, \bm\Sigma_f)_{\Omega_f}(s) \, ds \\
- \langle \bm G_f, \bm{\Sigma}_f \bm n_p\rangle_{\Gamma_I\cup\Gamma_f^D}(t)
+ \langle \bm G_f, \bm{\Sigma}_f \bm n_p\rangle_{\Gamma_I\cup\Gamma_f^D}(0) + 
\int_0^t \langle \dot{\bm G}_f, \bm{\Sigma}_f \bm n_p \rangle_{\Gamma_I\cup\Gamma_f^D} (s) \, ds. 
\end{multline*}
Applying again Cauchy-Schwarz and Young inequalities, using the trace inequality in $\mathbb{H}({\rm div},\Omega_f)$ together with  \cite[Lemma 2.2]{cancrini2024},
and recalling that $\boldsymbol{\Sigma}_f|_{t=0}={\bm 0}$, we obtain
\begin{multline*}
\int_0^t \mathcal{F}(\dot{\bm{u}}_p,\dot{\bm{w}}_p,\dot{\bm{\Sigma}}_f)(s) \, ds \lesssim \int_0^t (\| \bm f_p(s) \|_{\Omega_p}
+ \| \bm g_p(s) \|_{\Omega_p})
\|(\bm{u}_p,\bm{w}_p) (s) \|_{\mathcal{E}_p}\, ds \\
+ \frac{1}{2\epsilon} \|\bm F_f(t)\|^2_{\Omega_f} + \frac{1}{2\epsilon} \|\bm G_f(t)\|_{\Gamma_I\cup\Gamma_f^D}^2 + \epsilon\|\bm \Sigma_f(t)\|_{\mathcal{E}_f}^2  
\\  + \int_0^t \|\dot{\bm F}_f(s)\|_{\Omega_f}  \|\bm\Sigma_f(s)\|_{\mathcal{E}_f} + \|\dot{\bm G}_f(s)\|_{\Gamma_I\cup\Gamma_f^D}  \|\bm\Sigma_f(s)\|_{\mathcal{E}_f} \, ds.
\end{multline*}
Plugging the above estimate into \eqref{eq:partial-stability-estimate} and choosing $\epsilon$ small enough, we get 
\begin{multline*}
 \| (\bm u, \bm w)(t) \|^2_{\mathcal{E}_p} + \| \bm \Sigma (t) \|^2_{\mathcal{E}_f}
 \lesssim  \mathcal{G}_0^2
  + 
 \int_0^t \mathcal{N}(\bm f_p, \bm g_p, \bm F_f, \bm G_f) \left(  \| (\bm u, \bm w)(t) \|_{\mathcal{E}_p} + \| \bm \Sigma (t) \|_{\mathcal{E}_f} \right) \, ds
\end{multline*}
with $\mathcal{G}_0$ and $\mathcal{N}$ defined as in \eqref{eq:G_def} and \eqref{eq:N_def}, respectively.
%
Finally, we prove the assertion by taking the supremum over $t\in (0,T]$ and applying Gronwall's Lemma.
\end{proof}

\section{Discontinuous Galerkin space discretization}\label{sec:polydg_discr}
In this section, we present the PolydG discretization of the coupled problem consisting of systems \eqref{eq:biot-system}, \eqref{eq:stokes-sigma} and \eqref{eq:coupling-conditions}.

\subsection{Preliminaries}
We introduce a polytopic mesh $\mathcal{T}_h$ made of general polygons or polyhedra in two or three dimensions, respectively, and define $\mathcal{T}_h$ as $\mathcal{T}_h=\mathcal{T}^p_h \cup\mathcal{T}^f_h$, where $\mathcal{T}^{\diamond}_h=\{K\in\mathcal{T}_h:\overline{K}\subseteq\overline{\Omega}_{\diamond}\}$, with $\diamond=\{p,f\}$. We assume that the meshes $\mathcal{T}_h^p$  and $\mathcal{T}_h^f$ are aligned with $\Omega_p$ and $\Omega_f$, respectively.
We set the polynomial degrees $p_{p,K}\geq 1$ and  $p_{f,K}\geq 1$ in each mesh element of $\mathcal{T}_h^p$ and $\mathcal{T}_h^f$. 
The discrete polynomial spaces are introduced as follows: $$\bm{V}_h^p=[\mathcal{P}_{p_p}(\mathcal{T}_h^p)]^d, \; \bm{S}_h^f=[\mathcal{P}_{p_f}(\mathcal{T}_h^f)]^{d\times d} \; {\rm and} 
 \; \bm{\Lambda}_h^f=[\mathcal{P}_{p_f-1}(\mathcal{T}_h^f)]^{d^*}.  $$  
Moreover, $\mathcal{P}_{r}(\mathcal{T}^{\diamond}_h)$ is the space of piecewise polynomials in $\Omega_{\diamond}$ of total degree less than or equal to $r$ in any $K\in\mathcal{T}_h^{\diamond}$ with $\diamond=\{p,f\}$.

In the following, we assume that the model parameters in \eqref{eq:biot-system} and \eqref{eq:stokes-sigma} are element-wise constant for all $K\in\mathcal{T}_h^p \cup \mathcal{T}_h^f$.
To deal with polygonal and polyhedral elements, we define \textit{element-interface}  the intersection of the $(d-1)$-dimensional faces of  
any two neighboring elements of $\mathcal{T}_h$. If $d=2$, an interface/face is a line segment and the set of all interfaces/faces is denoted by $\mathcal{F}_h$.
If $d=3$, an interface is a polygon we assume could be further decomposed into planar triangles collected in the set  $\mathcal{F}_h$.  
We decompose $\mathcal{F}_h$ as $\mathcal{F}_h=\mathcal{F}_{h}^I \cup \mathcal{F}_h^p \cup  \mathcal{F}_h^f$, where
$ \mathcal{F}_{h}^I = \{F\in\mathcal{F}_h:F\subset\partial K^p\cap\partial K^f,K^p\in\mathcal{T}_{h}^p,K^f\in\mathcal{T}_{h}^f\},$  and 
$\mathcal{F}_h^p$, and $\mathcal{F}_h^f$ denote all the faces of $\mathcal{T}_h^p$, and $\mathcal{T}_h^f$ respectively, not laying on $\Gamma_I$.
Finally, the faces of $\mathcal{T}_h^p$ and $\mathcal{T}_h^f$ can be further written as the union of \textit{internal} ($i$) and \textit{boundary} ($b$) faces, respectively, namely,
$\mathcal{F}^p_h=\mathcal{F}^{p,i}_h\cup\mathcal{F}^{p,b}_h$ and $\mathcal{F}^f_h=\mathcal{F}^{f,i}_h\cup\mathcal{F}^{f,b}_h$, where $\mathcal{F}^{\diamond,b}_h = \mathcal{F}^{\diamond,N}_h \cup \mathcal{F}^{\diamond,D}_h$, with $\diamond =\{p,f\}$, include both the edges where Neumann and Dirichlet conditions are imposed. Following \cite{CangianiDongGeorgoulisHouston_2017}, we next introduce the main assumption on $\mathcal{T}_h$. 
\\
\begin{definition}\label{def::polytopic_regular}
A mesh $\mathcal{T}_h$ is said to be \textit{polytopic-regular} if for any $ K \in \mathcal{T}_h$, there exists a set of non-overlapping $d$-dimensional simplices contained in $K$, denoted by $\{S_K^F\}_{F\subset{\partial K}}$, such that for any face $F\subset\partial K$, it holds $h_K\lesssim d|S_K^F| \, |F|^{-1}$.
\end{definition}

\medskip
\begin{assumption}\label{ass:mesh}
The mesh $\mathcal{T}_h$ satisfies the following assumptions:
\begin{itemize}
    \item[a)] The sequence of meshes $\{\mathcal{T}_h\}_h$ is assumed to be \textit{uniformly} polytopic regular in the sense of  Definition~\ref{def::polytopic_regular}.
    \item[b)] For any pair of neighboring elements $K^\pm\in\mathcal{T}_h^{\diamond}$, it holds $h_{K^+}\lesssim h_{K^-}\lesssim h_{K^+},\ \ p_{\diamond,K^+}\lesssim p_{\diamond,K^-}\lesssim p_{\diamond,K^+}$, with $\diamond=\{e,p\}$.
\end{itemize}
\end{assumption}
This will allow us to avoid technicalities in the following proofs. We remark that these assumptions do not restrict the number of faces per element or their measure relative to the diameter of the element they belong to as pointed out in \cite{CangianiDongGeorgoulisHouston_2017}.
Under 
Assumption~\ref{ass:mesh}, the following \textit{trace-inverse inequality} holds:
\begin{align}
& ||v||_{L^2(\partial K)}\lesssim ph_K^{-1/2}||v||_{L^2(K)}
&& \forall \ K\in\mathcal{T}_h \ \forall v \in \mathcal{P}_p(K).
\label{eq::traceinv}
\end{align}
Next, we  make the following assumption for later use.
\medskip
\begin{assumption}\label{ass:covering}
Any mesh $\mathcal{T}_h$ admits a covering $\mathcal{T}_{\S}$ in the sense of Definition \ref{def::polytopic_regular} such that 
i) $\max_{K\in\mathcal{T}_h} \textrm{card}\{K^\prime\in\mathcal{T}_h:K^\prime\cap\mathcal{K}\neq\emptyset,\ \mathcal{K}\in\mathcal{T}_{\S} \text{ s.t.} \ K\subset\mathcal{K}\}\lesssim 1$ and 
ii) $h_\mathcal{K}\lesssim h_K $ for each pair $K\in\mathcal{T}_h, \ \mathcal{K}\in\mathcal{T}_{\S}$ with $K\subset\mathcal{K}$.
\label{ass::2}
\end{assumption}

Finally, as in  \cite{Arnoldbrezzicockburnmarini2002}, for sufficiently piecewise smooth scalar-, vector- and tensor-valued fields $\psi$, $\bm{v}$ and $\bm{\tau}$, respectively, we define the averages and jumps on each \textit{element-interface}  $F\in\mathcal{F}_h^{p,i}\cup\mathcal{F}_h^{f,i}\cup  \mathcal{F}_h^{I}$ shared by the elements $K^{\pm}\in \mathcal{T}_h$ as follows:
\begin{align*}
\llbrace \psi \rrbrace &= \frac{\psi^++\psi^-}{2}, 
&   \llbracket\psi\rrbracket  &=  \psi^+\bm{n}^++\psi^-\bm{n}^-, & & \\
\llbrace \bm{v}\rrbrace  & = \frac{\bm{v}^+ + \bm{v}^-}{2}, &  
 \llbracket\bm{v}\rrbracket & = \bm{v}^+\otimes\bm{n}^++\bm{v}^-\otimes\bm{n}^-, 
 & \llbracket\bm{v}\rrbracket_{\bm n} & = \bm{v}^+\cdot\bm{n}^++\bm{v}^-\cdot\bm{n}^-, \\
 \llbrace \bm{\tau} \rrbrace & = \frac{\bm{\tau}^+ + \bm{\tau}^-}{2}, 
 & \llbracket\bm{\tau}\rrbracket & = \bm{\tau}^+\bm{n}^+ + \bm{\tau}^-\bm{n}^-, & &
\end{align*}
where $\otimes$ is the tensor product in $\mathbb{R}^3$, $\cdot^{\pm}$ denotes the trace on $F$ taken within $K^\pm$, and $\bm{n}^\pm$ is the outer normal vector to $\partial K^\pm$. Accordingly, on \textit{boundary} faces $F\in\mathcal{F}_h^{p,b}\cup\mathcal{F}_h^{f,b}$, we set
$
\llbracket\psi\rrbracket = \psi\bm{n},\
\llbrace\psi \rrbrace = \psi,\
\llbrace\bm{v}\rrbrace= \bm{v},\
\llbracket\bm{v}\rrbracket= \bm{v}\otimes\bm{n},\
\llbracket\bm{v}\rrbracket_{\bm n} =\bm{v}\cdot\bm{n},\
\llbrace\bm{\tau}\rrbrace= \bm{\tau}, \llbracket \bm \tau \rrbracket = \bm \tau \bm n.$

For later use, we also define $\nabla_h$ and $(\nabla_h \cdot)$ to be the broken gradient and divergence operators, respectively, 
set $\bm{\varepsilon}_h(\bm{v})=(\nabla_h \bm{v} + \nabla_h \bm{v}^T)/2$ and use the short-hand notation  
$(\cdot,\cdot)_{\Omega_{\diamond}}=\sum_{K\in\mathcal{T}_h^{\diamond}} \int_K\cdot$ and $\langle\cdot,\cdot\rangle_{\mathcal{F}_h^{\diamond}}=\sum_{F\in\mathcal{F}_h^{\diamond}}\int_F\cdot$ for $\diamond = \{p,f\}$. 
In the following, we assume that $\mathbb{C}$, $m$ and $\rho_f$ are element-wise constant and we define 
$\overline{\mathbb{C}}_K=(|\mathbb{C}^{1/2}|_2^2)_{|K}$, $\overline{m}_{K}=( m)_{|K}$ for all $K\in\mathcal{T}_h^p$ and $\overline{\rho}_{f,K}=\rho_{f|K}$ for all $K\in\mathcal{T}_h^f$.

\subsection{Semi-discrete PolydG formulation}
We define the discrete space $\bm V_h = \bm V_h^p \times \bm V_h^p \times \bm S_h^f \times \bm \Lambda_h^f$ and introduce the semi-discrete problem: $\forall t \in (0,T]$, find $(\bm u_{ph}, \bm w_{ph}, \bm{\Sigma}_{fh}, \bm r_{fh})(t) \in \bm V_h$ s. t. $\forall (\bm{v},\bm{z},\bm{\tau}, \bm \lambda)\in \bm V_h$ it holds
\begin{multline}\label{eq:weak-dg}
\mathcal{M}^p((\ddot{\bm u}_{ph},\ddot{\bm{w}}_{ph}),(\bm v,\bm z)) + 
\mathcal{M}^f(\dot{\bm{\Sigma}}_{fh},\bm{\tau})
+ \mathcal{D}^p(\dot{\bm{w}}_{ph}, \bm z) +  \mathcal{D}^f(\dot{\bm{\Sigma}}_{fh}, \bm \tau)\\ + \mathcal{A}^p_h((\bm{u}_{ph},\bm{w}_{ph}),(\bm{v},\bm{z})) 
+ \mathcal{A}_h^f(\bm{\Sigma}_{fh}, \bm{\tau})  + \mathcal{B}^f(\bm r_{fh}, \bm \tau) - \mathcal{B}^f(\bm \lambda, \dot{\bm \Sigma}_{fh})  \\
+ \mathcal{C}^{pf} ((\dot{\bm{u}}_{ph},\dot{\bm{w}}_{ph}),\bm{\tau})  - \mathcal{C}^{fp} (\dot{\bm{\Sigma}}_{fh},(\bm{v},\bm{z})) = \mathcal{F}(\bm{v},\bm{z},\bm{\tau}), 
\end{multline}
with initial conditions $\bm{u}_{ph}(0)=\bm u_{0h}, \dot{\bm{u}}_{ph}(0)=\bm{v}_{0h}, \bm{w}_{ph}(0)=\bm{w}_{0h}, \dot{\bm{w}}_{ph}(0)=\bm{z}_{0h}$ and $\bm{\Sigma}_{fh}(0)=\bm{0}$, where 
$\bm u_{0h}, \bm v_{0h}, \bm w_{0h}$ and $ \bm z_{0h}$ are the $\bm L^2$-orthogonal projection of the initial data in \eqref{eq:biot-system}.
The bilinear form $\mathcal{A}_h^p(\cdot, \cdot)$ can be splitted as
\begin{equation}\label{eq:def_Ah}
\mathcal{A}_h^p((\bm u,\bm w),(\bm v,\bm z)) = \mathcal{A}_h^e(\bm{u},\bm v) + \mathcal{B}_h^p(\beta\bm{u}+\bm{w},\beta\bm{v}+\bm{z})
\end{equation}
where for any $ \bm u,\bm w,\bm v,\bm z \in \bm V_h^p$ and for any $\bm \Sigma,\bm \tau \in  \bm S_h^f$ it holds
\begin{align}
\mathcal{A}_h^e(\bm{u},\bm{v}) & = ( \bm{\sigma}_{eh}(\bm u), \bm{\varepsilon}_h(\bm{v}))_{\Omega_p}  -  
         \langle \llbrace \bm{\sigma}_{eh}(\bm{u})\rrbrace,
         \llbracket \bm{v} \rrbracket\rangle_{\mathcal{F}_h^p} \nonumber
         \\ & \qquad \qquad \qquad \qquad -  
         \langle \llbracket \bm{u} \rrbracket, \llbrace \bm{\sigma}_{eh}(\bm{v})\rrbrace\rangle_{\mathcal{F}_h^p} +
         \langle \chi_e
         \llbracket \bm{u} \rrbracket, \llbracket \bm{v} \rrbracket\rangle_{\mathcal{F}_h^p}, \label{def:bilinear_poro_e}\\[1mm]
    \mathcal{B}_h^p(\bm{w},\bm{z}) & = ( m\nabla_h\cdot\bm{w}, \nabla_h\cdot\bm{z})_{\Omega_p}  -  
  \langle \llbrace m\nabla_h\cdot\bm{w} \rrbrace, \llbracket \bm{z} \rrbracket_{\bm n} \rangle_{\mathcal{F}_h^p}  
    \nonumber \\ & \qquad \qquad  \qquad \qquad -    \langle \llbracket \bm{w} \rrbracket_{\bm n}, \llbrace m \nabla_h\cdot\bm{z} \rrbrace\rangle_{\mathcal{F}_h^p} + \langle \chi_p
         \llbracket \bm{w} \rrbracket_{\bm n}, \llbracket \bm{z} \rrbracket_{\bm n} \rangle_{\mathcal{F}_h^p}, \label{def:bilinear_poro_p} \\[1mm]
      \mathcal{A}_h^f(\bm{\Sigma},\bm{\tau}) & = (\rho_f^{-1} \nabla_h\cdot\bm{\Sigma}, \nabla_h\cdot\bm{\tau})_{\Omega_f}  -  \ \langle \llbrace \rho_f^{-1} \nabla_h \cdot \bm{\Sigma} \rrbrace, \llbracket \bm{\tau} \rrbracket \rangle_{\mathcal{F}_h^f} \nonumber \\ & \qquad \qquad\qquad\qquad - 
      \langle \llbracket \bm{\Sigma} \rrbracket, \llbrace \rho_f^{-1} \nabla_h\cdot\bm{\tau}\rrbrace \rangle_{\mathcal{F}_h^f}  +  \langle \chi_f\llbracket \bm{\Sigma} \rrbracket, \llbracket \bm{\tau} \rrbracket\rangle_{\mathcal{F}_h^f}. \label{def:bilinear_fluid}
\end{align}

Finally, we define the penalization functions $\chi_e, \chi_p\in L^\infty(\Omega_p)$ and $\chi_f \in L^\infty(\Omega_f)$ appearing in \eqref{def:bilinear_poro_e},\eqref{def:bilinear_poro_p} and \eqref{def:bilinear_fluid}, respectively:

\begin{equation}\label{eq:stab_term_1}
\chi_e|_F =
   \begin{cases}
     c_1 \max\limits_{{K \in \{K^+, K^-\}}}
         \overline{\mathbb{C}}_K \, p_{p,K}^2 h_{K}^{-1} &
     \forall F \in \mathcal{F}_h^{p,i}, \quad 
     F \subseteq \partial K^+ \cap \partial K^-,\\
     \overline{\mathbb{C}}_K \, p_{p,K}^2 h_{K}^{-1} & 
     \forall F \in \mathcal{F}_h^{p,b}, \quad 
     F \subseteq \partial K,
   \end{cases}
\end{equation}
\begin{equation}\label{eq:stab_term_2}
\chi_p|_F =
   \begin{cases}
     c_2 \max\limits_{{K \in \{K^+, K^-\}}}
     \overline{m}_K \, p_{p,K}^2 h_{K}^{-1} & 
     \forall F \in \mathcal{F}_h^{p,i}, \quad 
     F \subseteq \partial K^+ \cap \partial K^-,\\
     \overline{m}_K \, p_{p,K}^2 h_{K}^{-1} &
     \forall F \in \mathcal{F}_h^{p,b}, \quad 
     F \subseteq \partial K,
   \end{cases}
\end{equation}

\begin{equation}\label{eq:stab_term_3}
\chi_f|_F =
   \begin{cases}
     c_3 \max\limits_{{K \in \{K^+, K^-\}}}
     (\overline{\rho_f}_K)^{-1}  p_{f,K}^2 h_{K}^{-1} &
     \forall F \in F_h^{f,i}, \quad 
     F \subseteq \partial K^+ \cap \partial K^-\\
     (\overline{\rho_f}_K)^{-1}  p_{f,K}^2 h_{K}^{-1} &
     \forall F \in \mathcal{F}_h^{p,b}, \quad 
     F \subseteq \partial K,
   \end{cases}
\end{equation}
with $c_1$, $c_2$, $c_3 > 0$ positive constants to be suitably chosen. 

\subsection{Stability and semi-discrete error analysis}

To carry out the stability analysis of the semi-discrete problem \eqref{eq:weak-dg}, we introduce the energy norm defined for any $(\bm u,\bm w) \in C^1((0,T];\bm{V}_h^p \times \bm{V}_h^p)$ and $\bm \Sigma \in C^0((0,T];\bm{V}_h^f)$ as
\begin{equation}\label{eq::norm_def1}
\| (\bm u, \bm w, \bm \Sigma) (t)\|^2_{\rm E }  =    \| (\bm u, \bm w)(t) \|^2_{{\rm E}_p} + \| \bm \Sigma (t) \|^2_{{\rm E_f}},
\end{equation}
with
\begin{align*}
    \| (\bm u, \bm w)(t) \|^2_{\rm{E}_p} & =  \|  \dot{\bm u}(t) \|^2_{\Omega_p} +  \|  \dot{\bm w}(t) \|^2_{\Omega_p} +  \| (\eta/ k)^{\frac12} {\bm{w}} (t) \|^2_{\Omega_p} \\ & \qquad  \qquad\qquad\qquad + \textcolor{black}{\|\gamma^\frac12 {\bm{w}} \cdot \bm{n}_p(t)\|_{\Gamma_I}^2}   
  +  \|\bm u\|_{\rm dG,e}^2 + | (\beta \bm u + \bm w)(t) |_{\rm dG,p}^2,  \\
\| \bm \Sigma (t) \|^2_{{\rm E_f}} & = 
  \| (1/2\mu_f)^{\frac12} {\rm dev}({\bm{\Sigma}}) (t)\|^2_{\Omega_f} 
 +  |  \bm \Sigma (t) |^2_{ \rm dG, f} +  \textcolor{black}{\| \delta^{-\frac12} {\bm{\Sigma}} \bm{n}_p \wedge \bm{n}_p (t)\|_{\Gamma_I}^2,}
\end{align*}
and where
\begin{align*}
\|\bm{v}\|_{\rm dG,e}^2 & =\|\mathbb{C}^{1/2}\bm{\varepsilon}_h(\bm{v})\|_{\Omega_p}^2+\|\chi_e^{1/2}\llbracket\bm{v}\rrbracket\|_{\mathcal{F}_h^p\cup \mathcal{F}_h^{p,D}}^2  & \forall \bm v\in \bm V_h^p, \\
|\bm{z}|_{\rm dG,p}^2 &=\|{m}^{1/2}\nabla_h\cdot\bm z\|_{\Omega_p}^2+
\|\chi_p^{1/2}\llbracket\bm{z}\rrbracket_{\bm n}\|_{\mathcal{F}_h^p \cup \mathcal{F}_h^{p,D}}^2 & \forall \bm z\in \bm V_h^p,\\
    |\boldsymbol{\sigma}|_{\rm dG, f}^2 & = \| (\rho_f)^{-\frac12} \nabla_h \cdot \boldsymbol{\sigma} \|^2_{\Omega_f} + \| \ \chi_f^{1/2}\llbracket \boldsymbol{\sigma}\rrbracket\ \|^2_{\mathcal{F}_h^f \cup \mathcal{F}_h^{f,N}} & \forall \boldsymbol{\sigma} \in \boldsymbol{S}_h^f,
\end{align*}
with $\chi_e, \chi_p$ and $\chi_f$ defined as in \eqref{eq:stab_term_1}, \eqref{eq:stab_term_2}, and \eqref{eq:stab_term_3}, respectively.
For later use, we also define the following augmented norm or any $(\bm u,\bm w) \in C^1((0,T]; \bm H^2(\mathcal{T}_h^p) \times \bm H^2(\mathcal{T}_h^p))$ and any $\bm \Sigma \in C^0((0,T];\mathbb{ H}^2(\mathcal{T}_h^f))$ as
\begin{equation*}\label{eq::norm_def1_agm}
\trinorm{(\bm u, \bm w, \bm \Sigma)}^2_{\rm E }  =    \trinorm{ (\bm u, \bm w)(t) }^2_{{\rm E}_p} + \trinorm{ \bm \Sigma (t) }^2_{{\rm E_f}},
\end{equation*}
with
\begin{align*}
    \trinorm{ (\bm u, \bm w)(t) }^2_{\rm{E}_p} & =  \|  \dot{\bm u}(t) \|^2_{\Omega_p} +  \|  \dot{\bm w}(t) \|^2_{\Omega_p} +  \| (\eta/ k)^{\frac12} {\bm{w}} (t) \|^2_{\Omega_p} \\ & \qquad  \qquad\qquad\qquad + \textcolor{black}{\|\gamma^\frac12 {\bm{w}} \cdot \bm{n}_p(t)\|_{\Gamma_I}^2}   
  +  \trinorm{\bm u(t)}_{\rm dG,e}^2 + \trinorm{ (\beta \bm u + \bm w)(t)} _{\rm dG,p}^2,  \\
\trinorm{ \bm \Sigma (t) }^2_{{\rm E_f}} & = 
  \| (1/2\mu_f)^{\frac12} {\rm dev}({\bm{\Sigma}}) (t)\|^2_{\Omega_f} 
 +  \trinorm{ \bm \Sigma (t) }^2_{ \rm dG, f} +  \| \delta^{-\frac12} {\bm{\Sigma}} \bm{n}_p \wedge \bm{n}_p (t)\|_{\Gamma_I}^2,
\end{align*}
and where
\begin{align*}
\trinorm{\bm{v}}_{\rm dG,e}^2 &=\norm{\bm{v}}_{\rm dG,e}^2+\norm{\chi_e^{-1/2}\llbrace\mathbb{C}\bm{\varepsilon}_h(\bm{v})\rrbrace}_{\mathcal{F}_h^p}^2
&& \forall \bm v\in \bm H^2(\mathcal{T}_h^p),\\
\trinorm{\bm{z}}_{\rm dG,p}^2 &= |\bm{z}|_{\rm dG,p}^2+
\norm{\chi_p^{-1/2}\llbrace m \nabla_h\cdot\bm z \rrbrace}_{\mathcal{F}_h^p\cup \mathcal{F}_h^I}^2
&& \forall \bm z\in \bm H^2(\mathcal{T}_h^p),\\
   \trinorm{\boldsymbol{\sigma}}_{\rm dG,f}^2 & = |\boldsymbol{\sigma}|_{\rm dG,f}^2 +  \| \chi_f^{-1/2}\llbrace \bm{\nabla}_h \cdot \boldsymbol{\sigma}\rrbrace  \|^2_{\mathcal{F}_h^f} && \forall \boldsymbol{\sigma} \in \mathbb{H}^2(\mathcal{T}_h^f).
\end{align*}

The following Lemma establishes the coercivity and boundedness of the discrete bilinear forms $\mathcal{A}_h^e, \mathcal{B}_h^p$ and $\mathcal{A}_h^f$ defined in \eqref{def:bilinear_poro_e},\eqref{def:bilinear_poro_p} and \eqref{def:bilinear_fluid}, respectively.
\begin{lemma}\label{lem::cont_coerc}
Let Assumption \ref{ass:mesh} be satisfied. Then, it holds
\begin{align*}
& \mathcal{A}_h^e(\bm u,\bm v)\lesssim 
\|\bm u\|_{\rm dG,e}
\|\bm v\|_{\rm dG,e}
&&  \mathcal{A}_h^e(\bm u,\bm u)\gtrsim 
\|\bm u\|_{\rm dG,e}^2
&& \forall \bm u,\bm v\in\bm V_h^p,\\
& \mathcal{B}_h^p(\bm u,\bm v)\lesssim 
|\bm u|_{\rm dG,p}
|\bm v|_{\rm dG,p}
&& 
\mathcal{B}_h^p(\bm u,\bm u)\gtrsim 
|\bm u|_{\rm dG,p}^2
&& \forall \bm u,\bm v\in\bm V_h^p,\\
& \mathcal{A}_h^f(\boldsymbol{\sigma}, \boldsymbol{\tau})  \lesssim |\boldsymbol{\sigma}_h|_{\rm dG,f}
|\boldsymbol{\tau}_h|_{\rm dG,f}
&& \mathcal{A}_h^f(\boldsymbol{\sigma}, \boldsymbol{\sigma})  \gtrsim |\boldsymbol{\sigma}_h|_{\rm dG,f}^2  && \forall \boldsymbol{\sigma},\boldsymbol{\tau} \in \boldsymbol{S}_h^f,\\
& \mathcal{A}_h^e(\bm u,\bm v)\lesssim
\trinorm{\bm u}_{\rm dG,e}
\|\bm v\|_{\rm dG,e}
&&\qquad\forall \bm u \in\bm H^2(\mathcal{T}_h^p) 
&&\forall \bm v \in \bm{V}_h^p,\\
&\mathcal{B}_h^p(\bm w,\bm z)\lesssim
\trinorm{\bm w}_{\rm dG,p}
|\bm z|_{\rm dG,p}
&&\qquad\forall \bm w \in\bm H^2(\mathcal{T}_h^p)
&&\forall \bm z \in \bm{V}_h^p,\\
& \mathcal{A}_h^f(\boldsymbol{\sigma}, \boldsymbol{\tau})  \lesssim \trinorm{\boldsymbol{\sigma}}_{\rm dG,f} |\boldsymbol{\tau}|_{\rm dG,f}  && 
\qquad  \forall \boldsymbol{\sigma} \in \mathbb{H}^2(\mathcal{T}_h^f) &&  \forall \boldsymbol{\tau} \in \boldsymbol{S}_h^f.     
\end{align*}
The coercivity bounds hold provided that the stability parameters $c_1,c_2$ and $c_3$  in \eqref{eq:stab_term_1}, \eqref{eq:stab_term_2} and \eqref{eq:stab_term_3}, respectively, are chosen sufficiently large.
\end{lemma}
\begin{proof}
The proof is based on employing 
the same arguments as in \cite[Lemma A.3]{AntoniettiMazzieriNatipoltri2021} and in 
\cite[Lemma 3]{cancrini2024}.
\end{proof}

\begin{theorem}
For any time $t\in (0,T]$ let $(\bm u_{ph}, \bm w_{ph}, \bm \Sigma_{fh}, \bm r_{fh})(t) \in \bm V_h$ the solution to problem \eqref{eq:weak-dg}. Then, it holds

\begin{equation*}
\sup_{t\in(0,T]}\| (\bm u_{ph}, \bm w_{ph}, \bm \Sigma_{fh}) (t)\|_{\rm{E}}
 \lesssim  \mathcal{G}_{0h}+ 
 \int_0^T \mathcal{N}(\bm f_p, \bm g_p, \bm F_f, \bm G_f)(s) \, ds
\end{equation*}
where 
\begin{align}
    \mathcal{G}_{0h}^2 & = 
 \| (\bm u_{ph}, \bm w_{ph}, \bm \Sigma_{fh}) (0)\|_{\mathcal{E}}^2 
 + \sup_{t\in[0,T]}\left( \|\bm F_f(t)\|^2_{\Omega_f} +  \|\bm G_f(t)\|_{\Gamma_I\cup\Gamma_f^D}^2\right), \label{eq:Gh_def}
\end{align}
and $\mathcal{N}$ is defined as in \eqref{eq:N_def}.
\end{theorem}
\begin{proof}
The assertion follows the lines for the proof of Theorem \ref{teo:stability} and uses the results in Lemma \ref{lem::cont_coerc}. 
\end{proof}

\subsection{Error analysis}
In this section, we prove an a-priori error estimate in the energy norm \eqref{eq::norm_def1} for the semi-discrete problem \eqref{eq:weak-dg}. 
We start by introducing the following notation
 for any time $t\in (0,T]$,
\begin{align*}
    \bm e^u(t) & = (\bm u_p - \bm u_{ph})(t)   
 & = & \; (\bm u_p - \bm u_{pI})(t) &+&\;(\bm u_{pI} - \bm u_{ph})(t) & = &\; \bm e_I^u(t) - \bm e_h^u(t),\\
    \bm e^w(t) & = (\bm w_p - \bm w_{ph})(t)  & = &\; (\bm w_p - \bm w_{pI})(t) &+ &\; (\bm w_{pI} - \bm w_{ph})(t) & = &\; \bm e_I^w(t) - \bm e_h^w(t),\\
    \bm e^\Sigma(t) & = (\bm \Sigma_f - \bm \Sigma_{fh})(t)  & = &\; (\bm \Sigma_f - \bm \Sigma_{fI})(t) &+&\; (\bm \Sigma_{fI} - \bm \Sigma_{fh})(t) & = &\; \bm e_I^\Sigma(t) - \bm e_h^\Sigma(t), \\
    \bm e^r(t) & = (\bm r_f - \bm r_{fh})(t)  & = &\; (\bm r_f - \bm r_{fI})(t) &+&\; (\bm r_{fI} - \bm r_{fh})(t) & = &\; \bm e_I^r(t) - \bm e_h^r(t),
\end{align*}
and observe that \eqref{eq:weak-dg} is \textit{strongly consistent} in the sense that the error equation reads as follows for any $(\bm{v},\bm{z},\bm{\tau}, \bm \lambda)\in \bm V_h$:
\begin{multline}\label{eq:error}
\mathcal{M}^p((\ddot{\bm e}^u,\ddot{\bm{e}}^w),(\bm v,\bm z)) + 
\mathcal{M}^f(\dot{\bm{e}}^\Sigma,\bm{\tau})
+ \mathcal{D}^p(\dot{\bm{e}}^w, \bm z) +  \mathcal{D}^f(\dot{\bm{e}}^\Sigma, \bm \tau) + \mathcal{A}^p_h((\bm{e}^u,\bm{e}^w),(\bm{v},\bm{z})) \\
+ \mathcal{A}_h^f(\bm{e}^\Sigma, \bm{\tau})  + \mathcal{B}^f(\bm e^r, \bm \tau) - \mathcal{B}^f(\bm \lambda, \dot{\bm e}^\Sigma)  
+ \mathcal{C}^{pf} ((\dot{\bm{e}}^u,\dot{\bm{e}}^w),\bm{\tau})  - \mathcal{C}^{fp} (\dot{\bm{e}}^\Sigma,(\bm{v},\bm{z})) = 0. 
\end{multline}
For an open bounded polytopic domain $\Upsilon\subset\mathbb{R}^d$ and a generic polytopic mesh $\mathcal{T}_h$ over $\Upsilon$ satisfying Assumption \ref{ass::2},
as in  \cite{cangiani2014hp}, we can introduce the Stein extension operator $\tilde{\mathcal{E}}:H^m(\kappa)\rightarrow H^m(\mathbb{R}^d)$ \cite{stein1970singular}, for any $\kappa\in\mathcal{T}_h$ and $m\in\mathbb{N}_0$, such that $\tilde{\mathcal{E}}v|_\kappa=v$ and $\norm{\tilde{\mathcal{E}}v}_{m,\mathbb{R}^d}\lesssim\norm{v}_{m,\kappa}$. The corresponding vector-valued and tensor valued versions mapping $\bm H^m(\kappa)$ and $\mathbb{H}^m(\kappa)$ onto $\bm H^m(\mathbb{R}^d)$ and $\mathbb{H}^m(\mathbb{R}^d)$ act component-wise and are denoted in the same way. 
In what follows, for any $\kappa\in\mathcal{T}_h$, we will denote by $\mathcal{K}_\kappa$ the simplex belonging to the covering $\mathcal{T}_{\S}$ such that $\kappa\subset\mathcal{K}_\kappa$, cf. Assumption \ref{ass::2}.
\\
The next Lemma provides the interpolation bounds that are instrumental for the derivation of the a-priori error estimate.
\\
\begin{lemma}
For any $(\bm u,\bm w) \in  C^1([0,T];\, \bm H^m(\mathcal{T}_h^p)\times \bm H^\ell(\mathcal{T}_h^p))$, with $m,\ell\geq 2$, there exists $(\bm u_I,\bm w_I)\in C^1([0,T];\bm{V}_h^p\times \bm{V}_h^p)$ s.t.:
\begin{equation}\label{eq:interp_est}
\begin{aligned}
\trinorm{(\bm u-\bm u_I,\bm w-\bm w_I)}_{\rm E_p }^2 \lesssim&
\sum_{\kappa\in\mathcal{T}_h^p}{\frac{h_\kappa^{2(s_\kappa-1)}}{p_{p,\kappa}^{2m-3}}}\left(\norm{\widetilde{\mathcal{E}}\dot{\bm u}}_{m,\mathcal{K}_\kappa}^2+\norm{\widetilde{\mathcal{E}}\bm u}_{m,\mathcal{K}_\kappa}^2\right)\\+&
\sum_{\kappa\in\mathcal{T}_h^p}{\frac{h_\kappa^{2(r_\kappa-1)}}{p_{p,\kappa}^{2\ell-3}}}\left(\norm{\widetilde{\mathcal{E}}\dot{\bm w}}_{\ell,\mathcal{K}_\kappa}^2+\norm{\widetilde{\mathcal{E}}\bm w}_{\ell,\mathcal{K}_\kappa}^2\right).
\end{aligned}
\end{equation}
where $s_\kappa = \min(m,p_{p,\kappa}+1)$, and  $r_\kappa = \min(\ell,p_{p,\kappa}+1)$.
Also, for any  $\bm \Sigma \in  C^0([0,T];\, \mathbb{H}^n(\mathcal{T}_h^f))$, with $n\geq 2$, there exists $\bm \Sigma_I \in C^0([0,T];\bm{S}_h^f)$ s.t.:
\begin{equation}\label{eq:interp_est_sigma}
\begin{aligned}
\trinorm{\bm \Sigma-\bm \Sigma_I}_{\rm E_f}^2 \lesssim&
\sum_{\kappa\in\mathcal{T}_h^f}{\frac{h_\kappa^{2(q_\kappa-1)}}{p_{f,\kappa}^{2n-3}}} \norm{\widetilde{\mathcal{E}}\bm \Sigma }_{n,\mathcal{K}_\kappa}^2,
\end{aligned}
\end{equation}
where $q_\kappa = \min(n,p_{f,\kappa}+1)$.
Moreover, for any  $\bm r \in  C^0([0,T];\, \bm H^\nu(\mathcal{T}_h^f))$, with $\nu\geq 2$, there exists $\bm r_I \in C^0([0,T];\boldsymbol{\Lambda}_h^f)$ s.t.:
\begin{equation}\label{eq:interp_est_sigma}
\begin{aligned}
\|\bm r-\bm r_I\|_{\Omega_f}^2 \lesssim&
\sum_{\kappa\in\mathcal{T}_h^f}{\frac{h_\kappa^{2\zeta_\kappa}}{(p_{f,\kappa}-1)^{2\nu}}} \norm{\widetilde{\mathcal{E}}\bm r }_{\nu,\mathcal{K}_\kappa}^2,
\end{aligned}
\end{equation}
where $\zeta_\kappa = \min(\nu,p_{f,\kappa})$.\label{lem::interp2}
\end{lemma}

\begin{proof}
    We prove \eqref{eq:interp_est} by combining the results in \cite[Lemma 33]{CangianiDongGeorgoulisHouston_2017} with  the ones in \cite[Lemma 4.2] {AntoniettiMazzieriNatipoltri2021} while we obtain \eqref{eq:interp_est_sigma} by combining again the results in \cite[Lemma 33]{CangianiDongGeorgoulisHouston_2017} with the ones in \cite[Lemma 4.1]{cancrini2024}.
\end{proof}

In addition to the continuity and coercivity results of Lemma \ref{lem::cont_coerc}, the a-priori error analysis requires an inf-sup condition for the constraint form $\mathcal{B}^f(\cdot,\cdot)$, as follows:
\begin{assumption}\label{hp:infsup}
    There exist a constant $\beta_h^f>0$ such that the following inequality holds :
\begin{equation}\label{eq:infsup}
\sup_{\boldsymbol{\tau}\in\boldsymbol{S}_h^f\setminus\{\boldsymbol{0}\}}\frac{\mathcal{B}^f(\boldsymbol{\lambda}_h,\boldsymbol{\tau})}{\|\boldsymbol{\tau}\|_{{\rm E}_f} +\| \ \chi_f^{1/2}\llbracket \boldsymbol{\sigma}\rrbracket\ \|^2_{\mathcal{F}_h^I}} \geq \beta_h^f\|\boldsymbol{\lambda}_h\|_{\Omega_f},
\qquad \forall \,\boldsymbol{\lambda}_h\in\boldsymbol{\Lambda}_h^f.
\end{equation}
\end{assumption}
Although the previous result is not available for polygonal meshes (and will be the subject of future work), it can be proven for matching simplicial meshes (that coincide with their covering $\mathcal{T}_{\S}$ defined in Assumption \ref{ass:covering}). In this case, the proof of the inf-sup inequality \eqref{eq:infsup} is based on \cite{boffi2009reduced}, with modifications to account for the interface terms in the norm $\|\boldsymbol{\tau}\|_{\text{E}_f}$ and in the additional term in the denominator, and it is reported in Appendix \ref{sec:inf-sup}. 

Under the previous assumption, we can establish the instrumental result:
\begin{lemma}\label{lem:inf-sup}
Let Assumption \ref{hp:infsup} be verified.
Then, 
 the following holds: 
 \begin{multline}\label{eq:bound-rh}
    \beta^f_h\|{\bm e}^{\bm r}_h\|_{\Omega_f}
    \lesssim \\ 
    \|\dot{\bm e}^{\boldsymbol{\Sigma}}_h\|_{\rm E_f}
    + \|{\bm e}^{\boldsymbol{\Sigma}}_h\|_{\rm E_f}
    +\|(\dot{\bm e}^{\bm u}_h,\dot{\bm e}^{\bm w}_h)\|_{\rm E_p}
    + \trinorm{\dot{\bm e}^{\boldsymbol{\Sigma}}_I}_{\rm E_f}
    + \trinorm{{\bm e}^{\boldsymbol{\Sigma}}_I}_{\rm E_f}
    +\|(\dot{\bm e}^{\bm u}_I,\dot{\bm e}^{\bm w}_I)\|_{\rm E_p}
    + \|{\bm e}^{\bm r}_I\|_{\Omega_f}
 \end{multline}
\end{lemma}
\begin{proof}
To prove \eqref{eq:bound-rh}, we start from the error equation \eqref{eq:error} with $\bm v=\bm 0, \bm z=\bm 0, \bm \lambda=\bm 0$ and a generic $\boldsymbol{\tau}\in{\bm S}_h^f$.
Rearranging the terms to isolate $\mathcal B^f({\bm e}^{\bm r}_h, {\bm \tau})$ on one side of the equality, using the continuity results of Lemmas \ref{lemma:stab-fluid} and \ref{lem::cont_coerc} and the trace-inverse inequality \eqref{eq::traceinv} on the interface term $\mathcal{C}^{pf}$, we obtain the following:
\begin{multline*}
    \mathcal B^f({\bm e}^{\bm r}_h, \boldsymbol{\tau})
    =
    \mathcal{M}^f(\dot{\bm{e}}^\Sigma,\bm{\tau})
+\mathcal{D}^f(\dot{\bm{e}}^\Sigma, \bm \tau) + \mathcal{A}_h^f(\bm{e}^\Sigma, \bm{\tau}) +
\mathcal{C}^{pf} ((\dot{\bm{e}}^u,\dot{\bm{e}}^w),\bm{\tau})
+ \mathcal B^f({\bm e}^{\bm r}_I, \boldsymbol{\tau}) \lesssim 
    \\
    \qquad (\|\dot{\bm e}_h^{\boldsymbol{\Sigma}}\|_{\text{E}_f} + \|{\bm e}_h^{\boldsymbol{\Sigma}}\|_{\text{E}_f} ) 
    \|{\boldsymbol{\tau}}\|_{\text{E}_f}
    +\| \ \chi_f^{1/2}\boldsymbol{\tau}\bm n_p \|^2_{\mathcal{F}_h^I}
    (\|\dot{\bm e}_h^{\bm u}\|_{\text{E}_p} + \|\dot{\bm e}_h^{\bm w}\|_{\text{E}_p}+\|\dot{\bm e}_I^{\bm u}\|_{\text{E}_p} + \|\dot{\bm e}_I^{\bm w}\|_{\text{E}_p})
    \\
    \quad + (\trinorm{\dot{\bm e}_I^{\boldsymbol{\Sigma}}}_{\text{E}_f} + \trinorm{{\bm e}_h^{\boldsymbol{\Sigma}}}_{\text{E}_f} 
    + \|{\bm e}^{\bm r}_I\|_{\Omega_f}) 
    \|{\boldsymbol{\tau}}\|_{\text{E}_f}. 
\end{multline*}
Taking the supremum over $\boldsymbol{\tau}\in\boldsymbol{S}_h^f\setminus\{\boldsymbol{0}\}$ and using the inf-sup condition \eqref{eq:infsup} completes the proof.
\end{proof}

We are now ready to state the main result of this section.
\begin{theorem}[A-priori error estimates] \label{thm::error-estimate}
Let Assumption \ref{ass:mesh} and \ref{ass::2}  and the hypothesis of Theorem \ref{teo:stability} hold and let the exact solution $(\bm u_p,\bm w_p,\bm \Sigma_f, \bm r_f)$ of problem \eqref{eq:weak-form} be such that 
\begin{multline*}
(\bm u_p, \bm w_p) \in C^2((0,T];\bm H^m(\mathcal{T}_h^p)\times\bm H^\ell(\mathcal{T}_h^p)) \,\cap\, C^1([0,T]; \bm{H}^1_0(\Omega_p)\times \bm H_{0,\Gamma_p^D}({\rm div},\Omega_p)),\\
\bm \Sigma_f \in C^1((0,T];\mathbb{H}^n(\mathcal{T}_h^f)) \,\cap\, C^0([0,T]; \mathbb{H}_{0,\Gamma_p^N}({\rm div},\Omega_f)), \; and \; \bm r_f \in  C^0((0,T];\, H^\nu(\mathcal{T}_h^f)^{d^*}),
\end{multline*}
 with $m,\ell,n,\nu \geq 2$ and let $(\bm u_{ph},\bm w_{ph},\bm \Sigma_{fh}, \bm r_{fh})$ such that 
\begin{multline*}
(\bm u_{ph}, \bm w_{ph}) \in C^2([0,T];\bm V_h^p \times\bm V_h^p), \; 
\bm \Sigma_{fh} \in C^1([0,T];\bm S_h^f), \; and \; \bm r_{fh} \in  C^0([0,T];\, \bm \Lambda_h^f),
\end{multline*}
 be the solution of the semi-discrete problem \eqref{eq:weak-dg}, with sufficiently large penalty parameters $c_1$, $c_2$ and $c_3$. 
Then, for any $t\in (0,t]$, the discretization error $\bm E(t)=(\bm e^u, \bm e^w, \bm e^\Sigma)(t)$ satisfies
\begin{align*}
\sup_{t\in(0,T]} \norm{\bm E(t)}_{\rm E}\lesssim &  \sum_{\kappa\in\mathcal{T}_h^p}\left({\frac{h_\kappa^{s_\kappa-1}}{p_{p,\kappa}^{m-3/2}}} \Theta_{\bm u} +
{\frac{h_\kappa^{r_\kappa-1}}{p_{p,\kappa}^{\ell-3/2}}}\Theta_{\bm w}\right) +
\sum_{\kappa\in\mathcal{T}_h^f}\left({\frac{h_\kappa^{q_\kappa-1}}{p_{f,\kappa}^{n-3/2}}} \Theta_{\bm \Sigma}
 + \frac{h_\kappa^{\zeta_\kappa}}{(p_{f,\kappa}-1)^\nu} \Theta_{\bm r}\right) 
\end{align*}
where
\begin{align*}
 \Theta_{\bm u}  & =  \sup_{t\in(0,T]} \left(\norm{\widetilde{\mathcal{E}}\dot{\bm u}(t)}_{m,\mathcal{K}_\kappa}+\norm{\widetilde{\mathcal{E}}\bm u(t)}_{m,\mathcal{K}_\kappa}\right) +
\int_0^T \left( \norm{\widetilde{\mathcal{E}}\ddot{\bm u}(s)}_{m,\mathcal{K}_\kappa}+\norm{\widetilde{\mathcal{E}}\dot{\bm u}(s)}_{m,\mathcal{K}_\kappa} \right) \,ds,  \\
 \Theta_{\bm w}  & =  \sup_{t\in(0,T]} \left(\norm{\widetilde{\mathcal{E}}\dot{\bm w}(t)}_{m,\mathcal{K}_\kappa}+\norm{\widetilde{\mathcal{E}}\bm w(t)}_{m,\mathcal{K}_\kappa}\right) +
\int_0^T \left( \norm{\widetilde{\mathcal{E}}\ddot{\bm w}(s)}_{m,\mathcal{K}_\kappa}+\norm{\widetilde{\mathcal{E}}\dot{\bm w}(s)}_{m,\mathcal{K}_\kappa} \right) \,ds,  \\
\Theta_{\bm \Sigma} & = \sup_{t\in(0,T]} \norm{\widetilde{\mathcal{E}}\bm \Sigma(t) }_{n,\mathcal{K}_\kappa} + \int_0^T \norm{\widetilde{\mathcal{E}}\dot {\bm \Sigma}(s) }_{n,\mathcal{K}_\kappa} \, ds,\\
\Theta_{\bm r} & = \int_0^T \norm{\widetilde{\mathcal{E}}\bm r(s) }_{\nu,\mathcal{K}_\kappa} \, ds,
\end{align*}
and where $s_\kappa = \min(m,p_{p,\kappa}+1), r_\kappa = \min(\ell,p_{p,\kappa}+1), q_\kappa = \min(n,p_{f,\kappa}+1)$ and $\zeta_\kappa = \min(\nu,p_{f,\kappa})$ for any $\kappa \in \mathcal{T}_h$.
Here the hidden constant depends on the material properties but is independent of the discretization parameters.
\end{theorem}

\begin{proof}
We consider equation \eqref{eq:error} for $\bm v = \dot{\bm e}_h^u, \bm z = \dot{\bm e}_h^w, \bm \tau = \dot{\bm e}_h^\Sigma$ and $\bm \lambda = \bm e_h^r$ to get 
\begin{multline*}
\circled{1} = \mathcal{M}^p((\ddot{\bm e}_{h}^u,\ddot{\bm{e}}_{h}^w),(\dot{\bm e}_h^u,\dot{\bm e}_h^w)) 
+ \mathcal{D}^p(\dot{\bm{e}}_{h}^w, \dot{\bm{e}}_{h}^w) +  \mathcal{A}^p_h((\bm{e}_{h}^u,\bm{e}_{h}^w),(\dot{\bm e}_h^u,\dot{\bm e}_h^w)) \\ 
+ 
\mathcal{M}^f(\dot{\bm e}_{h}^\Sigma,\dot{\bm e}_h^\Sigma) + \mathcal{D}^f(\dot{\bm e}_{h}^\Sigma, \dot{\bm e}_{h}^\Sigma) + \mathcal{A}_h^f(\bm{e}_{h}^\Sigma, \dot{\bm e}_{h}^\Sigma) 
= \\
\mathcal{M}^p((\ddot{\bm e}_{I}^u,\ddot{\bm{e}}_{I}^w),(\dot{\bm e}_h^u,\dot{\bm e}_h^w)) 
+ \mathcal{D}^p(\dot{\bm{e}}_{I}^w, \dot{\bm{e}}_{h}^w) + \mathcal{A}^p_h((\bm{e}_{I}^u,\bm{e}_{I}^w),(\dot{\bm e}_h^u,\dot{\bm e}_h^w)) 
\\ + \mathcal{M}^f(\dot{\bm e}_{I}^\Sigma,\dot{\bm e}_h^\Sigma)
 + \mathcal{D}^f(\dot{\bm e}_{I}^\Sigma, \dot{\bm e}_{h}^\Sigma) + \mathcal{A}_h^f(\bm{e}_{I}^\Sigma, \dot{\bm e}_{h}^\Sigma)   \\ + \mathcal{B}^f(\bm e_{I}^r, \dot{\bm e}_{h}^\Sigma) - \mathcal{B}^f(\bm e_h^r, \dot{\bm e}_{I}^\Sigma)  
+ \mathcal{C}^{pf} ((\dot{\bm{e}}_{I}^u,\dot{\bm{e}}_{I}^w),\dot{\bm e}_{h}^\Sigma)  - \mathcal{C}^{fp} (\dot{\bm e}_{I}^\Sigma,(\dot{\bm{e}}_{h}^u,\dot{\bm{e}}_{h}^w)) = \circled{2}.
\end{multline*}

By integrating $\circled{1}$ and $\circled{2}$ with respect to time in $(0,t)$ we obtain 

\begin{multline}\label{eq:error_lhs}
\int_0^t \circled{1} \,ds = \frac12 \mathcal{M}^p((\dot{\bm e}_{h}^u,\dot{\bm{e}}_{h}^w),(\dot{\bm e}_h^u,\dot{\bm e}_h^w)) 
+ \int_0^t \mathcal{D}^p(\dot{\bm{e}}_{h}^w, \dot{\bm{e}}_{h}^w)\,ds + \frac12  \mathcal{A}^p_h((\bm{e}_{h}^u,\bm{e}_{h}^w),(\bm e_h^u,\bm e_h^w)) \\ 
+ 
\int_0^t \mathcal{M}^f(\dot{\bm e}_{h}^\Sigma,\dot{\bm e}_h^\Sigma)\, ds + \int_0^t \mathcal{D}^f(\dot{\bm e}_{h}^\Sigma, \dot{\bm e}_{h}^\Sigma)\, ds + \frac12 \mathcal{A}_h^f(\bm{e}_{h}^\Sigma, \bm e_{h}^\Sigma),
\end{multline}
since $\bm e_h^u(0) = \bm e_h^w(0) = \dot{\bm e}_h^u(0) = \dot{\bm e}_h^w(0) = \bm e_h^\Sigma(0) = \bm 0$ and 
\begin{multline*}
\int_0^t \circled{2} \, ds = \int_0^t
\overbrace{    \Big( \mathcal{M}^p((\ddot{\bm e}_{I}^u,\ddot{\bm{e}}_{I}^w),(\dot{\bm e}_h^u,\dot{\bm e}_h^w)) 
+  \mathcal{D}^p(\dot{\bm{e}}_{I}^w, \dot{\bm{e}}_{h}^w) - \mathcal{A}^p_h((\dot{\bm{e}}_{I}^u,\dot{\bm{e}}_{I}^w),(\bm e_h^u,\bm e_h^w)) \Big)}^{T_1(s)} \, ds 
\\ + \overbrace{\mathcal{A}^p_h((\bm{e}_{I}^u,\bm{e}_{I}^w),(\bm e_h^u,\bm e_h^w))}^{T_2} \\ + \int_0^t \overbrace{\Big(\mathcal{M}^f(\dot{\bm e}_{I}^\Sigma,\dot{\bm e}_h^\Sigma)
 + \mathcal{D}^f(\dot{\bm e}_{I}^\Sigma, \dot{\bm e}_{h}^\Sigma) -\mathcal{A}_h^f(\dot{\bm{e}}_{I}^\Sigma, \bm e_{h}^\Sigma) \Big)}^{T_3(s)} \, ds + \overbrace{\mathcal{A}_h^f(\bm{e}_{I}^\Sigma, \bm e_{h}^\Sigma)}^{T_4}   \\ + \int_0^t \overbrace{\Big(\mathcal{B}^f(\bm e_{I}^r, \dot{\bm e}_{h}^\Sigma) - \mathcal{B}^f(\bm e_h^r, \dot{\bm e}_{I}^\Sigma)\Big)}^{T_5(s)}  \, ds \\
+ \int_0^t \Big(\overbrace{\mathcal{C}^{pf} ((\dot{\bm{e}}_{I}^u,\dot{\bm{e}}_{I}^w),\dot{\bm e}_{h}^\Sigma)}^{T_6(s)}  - \overbrace{\mathcal{C}^{fp} (\dot{\bm e}_{I}^\Sigma,(\dot{\bm{e}}_{h}^u,\dot{\bm{e}}_{h}^w))}^{T_7(s)} \Big) \, ds, 
\end{multline*}
respectively.
Next, we treat separately the terms $T_i$, $i=1,...,7$ as follows. For positive $\epsilon_i$, $i=1,...,4$, we employ Cauchy-Schwarz and Young's inequalities as follows  \begin{multline}\label{eq:T_1}
    \int_0^t T_1(s) \, ds + T_2 \lesssim \int_0^t \trinorm{(\dot{\bm e}_{I}^u,\dot{\bm{e}}_{I}^w)}_{\rm E_p}\norm{(\bm e_{h}^u,\bm{e}_{h}^w)}_{\rm E_p} \, ds \\
    + \frac{1}{2\epsilon_1} \int_0^t \mathcal{D}^p(\dot{\bm{e}}_{I}^w, \dot{\bm{e}}_{I}^w) \, ds +  \frac{\epsilon_1}{2} \int_0^t \mathcal{D}^p(\dot{\bm{e}}_{h}^w, \dot{\bm{e}}_{h}^w) \, ds \\
    + \frac{1}{2\epsilon_2} \mathcal{A}^p_h((\bm{e}_{I}^u,\bm{e}_{I}^w),(\bm e_I^u,\bm e_I^w))
    +  \frac{\epsilon_2}{2} \mathcal{A}^p_h((\bm{e}_{h}^u,\bm{e}_{h}^w),(\bm e_h^u,\bm e_h^w)),
\end{multline}
and
\begin{multline}
    \int_0^t T_3(s) \, ds + T_4 \lesssim  \frac{1}{2\epsilon_3} \int_0^t 
    \Big(\mathcal{M}^f(\dot{\bm e}_{I}^\Sigma,\dot{\bm e}_I^\Sigma) + \mathcal{D}^f(\dot{\bm e}_{I}^\Sigma,\dot{\bm e}_I^\Sigma)\Big) \, ds \\ + 
    \frac{\epsilon_3}{2} \int_0^t 
    \Big(\mathcal{M}^f(\dot{\bm e}_{h}^\Sigma,\dot{\bm e}_h^\Sigma) +     \mathcal{D}^f(\dot{\bm e}_{h}^\Sigma,\dot{\bm e}_h^\Sigma) \Big) \, ds 
    +  \int_0^t \trinorm{\dot{\bm e}^\Sigma_I}_{{\rm E_f}}\norm{\bm e^\Sigma_h}_{{\rm E_f}}  \, ds
    \\
    + \frac{1}{2\epsilon_4} \mathcal{A}_h^f(\bm{e}_{I}^\Sigma, \bm e_{I}^\Sigma)
    +  \frac{\epsilon_4}{2} \mathcal{A}_h^f(\bm{e}_{h}^\Sigma, \bm e_{h}^\Sigma).
\end{multline}

For $T_5$, we employ Cauchy-Schwarz and Young inequalities and then estimate \eqref{eq:bound-rh} of Lemma \ref{lem:inf-sup} to obtain, for positive $\epsilon_5,\epsilon_6$,
\begin{multline}\label{eq:T5-init}
    \int_0^tT_5(s)\,ds \lesssim
    \int_0^t\frac{1}{2\epsilon_5}\|{\bm e}_I^r\|_{\Omega_f}^2 + \int_0^t\frac{\epsilon_5}{2}\|\dot{\bm e}_h^{\bm \Sigma}\|_{\Omega_f}^2
    + \int_0^t\frac{1}{2\epsilon_6}\|\dot{\bm e}_I^{\bm \Sigma}\|_{\Omega_f}^2 + \int_0^t\frac{\epsilon_6}{2}\|{\bm e}_h^r\|_{\Omega_f}^2
    \\
    \lesssim
    \int_0^t\frac{1}{2\epsilon_5}\|{\bm e}_I^r\|_{\Omega_f}^2 + \int_0^t\frac{1}{2\epsilon_6}\|\dot{\bm e}_I^{\bm \Sigma}\|_{\Omega_f}^2
+    \int_0^t\left(\frac{\epsilon_5}{2}+\frac{\epsilon_6}{2(\beta_h^f)^2}\right)\|\dot{\bm e}_h^{\bm \Sigma}\|_{\rm E_f}^2
    \\+ 
      \int_0^t\frac{\epsilon_6}{2(\beta_h^f)^2}\|{\bm e}_h^{\boldsymbol{\Sigma}}\|_{\rm E_f}^2
    + \int_0^t\frac{\epsilon_6}{2(\beta_h^f)^2}\|(\dot{\bm e}_h^{\bm u},\dot{\bm e}_h^{\bm w})\|_{\rm E_p}^2
    \overset{\rm def}{=}
    \mathcal J_{\mathcal B}^t.
\end{multline}

Finally, for the coupling terms in $\mathcal{C}^{fp}(\cdot)$ we use the inverse inequality \eqref{eq::traceinv} together with Assumption \ref{ass:mesh} to get
\begin{multline}
\int_0^t T_7(s) \,ds 
 = \int_0^t \Big(<\dot{\bm{e}}_I^\Sigma \bm{n}_p \wedge \bm{n}_p, (\alpha \dot{\bm{e}}_{h}^u + \dot{\bm{e}}_{h}^w) \cdot \bm{n}_p>_{\Gamma_I} + <\dot{\bm{e}}_I^\Sigma \bm{n}_p \wedge \bm{n}_p, \dot{\bm{e}}_{h}^u \wedge \bm{n}_p>_{\Gamma_I} \Big) \, ds
 \\
 \lesssim
\int_0^t \sum_{\kappa_p\in\mathcal{T}_{h,I}^p,\,
\kappa_f\in\mathcal{T}_{h,I}^f} \norm{\dot{\bm{e}}_I^\Sigma}_{\partial \kappa_f}\Big(\norm{ \dot{\bm{e}}_{h}^u}_{\partial \kappa_p} + \norm{\dot{\bm{e}}_{h}^w}_{\partial \kappa_p} \Big) \, ds
\\ \lesssim
\int_0^t \Big(
 \sum_{\kappa \in\mathcal{T}_{h,I}^f} {p_{f,\kappa} h_{\kappa}^{-1/2}} \norm{\dot{\bm{e}}_I^\Sigma}_{\partial \kappa}\Big) \Big(\norm{ \dot{\bm{e}}_{h}^u}_{\Omega_p} + \norm{\dot{\bm{e}}_{h}^w}_{\Omega_p} \Big) \, ds\\
\stackrel{\text{def}}{=} \int_0^t \mathcal{I}_h^f(\dot{\bm e}_I^\Sigma)\Big(\norm{ \dot{\bm{e}}_{h}^u}_{\Omega_p} + \norm{\dot{\bm{e}}_{h}^w}_{\Omega_p} \Big) \, ds,
\end{multline}
being $\mathcal{T}_{h,I}^p$ and. $\mathcal{T}_{h,I}^f$ the sets of mesh elements sharing an edge with $\Gamma_I$. For $T_6$,
we use the integration by parts formula and we reason as before 
\begin{multline}\label{eq:T_6}
\int_0^t T_6(s) \,ds = \mathcal{C}^{pf} ((\dot{\bm{e}}_{I}^u,\dot{\bm{e}}_{I}^w),\bm e_{h}^\Sigma) - \int_0^t \mathcal{C}^{pf} ((\ddot{\bm{e}}_{I}^u,\ddot{\bm{e}}_{I}^w),\bm e_{h}^\Sigma) \, ds \\ 
\lesssim
 \sum_{\kappa \in\mathcal{T}_{h,p}^I} {p_{p,\kappa} h_{\kappa}^{-1/2}} \Big( \norm{ \dot{\bm{e}}_{I}^u}_{\partial \kappa} + \norm{\dot{\bm{e}}_{I}^w}_{\partial \kappa} \Big)
 \norm{\bm{e}_h^\Sigma}_{{\rm E_f}} \\
 + \int_0^t \sum_{\kappa \in\mathcal{T}_{h,p}^I} {p_{p,\kappa} h_{\kappa}^{-1/2}} \Big( \norm{ \ddot{\bm{e}}_{I}^u}_{\partial \kappa} + \norm{\ddot{\bm{e}}_{I}^w}_{\partial \kappa} \Big)
 \norm{\bm{e}_h^\Sigma}_{{\rm E_f}} \, ds
\\
 \stackrel{\text{def}}{=}  \mathcal{I}_h^p(\dot{\bm e}_I^u,\dot{\bm e}_I^w) \norm{\bm{e}_h^\Sigma}_{{\rm E_f}}  + \int_0^t \mathcal{I}_h^p(\ddot{\bm e}_I^u,\ddot{\bm e}_I^w) \norm{\bm{e}_h^\Sigma}_{{\rm E_f}} \, ds,
\end{multline}
where we also use the norm $\norm{\cdot}_{\rm E_f}$ to bound the $L^2$-norm $\norm{\cdot}_{\Omega_f}$.
Now, by putting together \eqref{eq:error_lhs} with \eqref{eq:T_1}--\eqref{eq:T_6} and choosing $\epsilon_i$ for $i=1,\ldots,6$, we obtain
\begin{multline}
\circled{3}  =  \mathcal{M}^p((\dot{\bm e}_{h}^u,\dot{\bm{e}}_{h}^w),(\dot{\bm e}_h^u,\dot{\bm e}_h^w)) 
+ \int_0^t \mathcal{D}^p(\dot{\bm{e}}_{h}^w, \dot{\bm{e}}_{h}^w)\,ds +  \mathcal{A}^p_h((\bm{e}_{h}^u,\bm{e}_{h}^w),(\bm e_h^u,\bm e_h^w)) \\ 
+ 
\int_0^t \mathcal{M}^f(\dot{\bm e}_{h}^\Sigma,\dot{\bm e}_h^\Sigma)\, ds + \int_0^t \mathcal{D}^f(\dot{\bm e}_{h}^\Sigma, \dot{\bm e}_{h}^\Sigma)\, ds + \mathcal{A}_h^f(\bm{e}_{h}^\Sigma, \bm e_{h}^\Sigma)
\\ \lesssim
\int_0^t \trinorm{(\dot{\bm e}_{I}^u,\dot{\bm{e}}_{I}^w)}_{\rm E_p}\norm{(\bm e_{h}^u,\bm{e}_{h}^w)}_{\rm E_p} \, ds 
    +  \int_0^t \mathcal{D}^p(\dot{\bm{e}}_{I}^w, \dot{\bm{e}}_{I}^w) \, ds  \\
+ \mathcal{A}^p_h((\bm{e}_{I}^u,\bm{e}_{I}^w),(\bm e_I^u,\bm e_I^w)) + 
 \int_0^t 
    \Big(\mathcal{M}^f(\dot{\bm e}_{I}^\Sigma,\dot{\bm e}_I^\Sigma) + \mathcal{D}^f(\dot{\bm e}_{I}^\Sigma,\dot{\bm e}_I^\Sigma)\Big) \, ds \\ 
    +  \int_0^t \trinorm{\dot{\bm e}^\Sigma_I}_{{\rm E_f}}\norm{\bm e^\Sigma_h}_{{\rm E_f}}  \, ds
    + \mathcal{A}_h^f(\bm{e}_{I}^\Sigma, \bm e_{I}^\Sigma)
     + \int_0^t \mathcal{I}_h^f(\dot{\bm e}_I^\Sigma)\Big(\norm{ \dot{\bm{e}}_{h}^u}_{\Omega_p} + \norm{\dot{\bm{e}}_{h}^w}_{\Omega_p} \Big) \, ds \\
      + \mathcal{I}_h^p(\dot{\bm e}_I^u,\dot{\bm e}_I^w) \norm{\bm{e}_h^\Sigma}_{{\rm E_f}}  + \int_0^t \mathcal{I}_h^p(\ddot{\bm e}_I^u,\ddot{\bm e}_I^w) \norm{\bm{e}_h^\Sigma}_{{\rm E_f}} \, ds
    {\color{black}+\mathcal J_{\mathcal B}^t}
      = \circled{4} + 
    {\color{black}\mathcal J_{\mathcal B}^t}.
\end{multline}
To bound $\circled{3}$ from below we reason as for the proof of Therorem \ref{teo:stability}  to have 
\begin{equation*}
    |\circled{3}| \gtrsim  \| (\bm e_h^u, \bm e_h^w)(t) \|^2_{{\rm E}_p} + \| \bm e_h^\Sigma (t) \|^2_{{\rm E_f}}.
\end{equation*}
Next, we rearrange the terms for  $\circled{4}$ and write
\begin{multline}\label{eq:circled_4}
\circled{4}  =
\int_0^t \trinorm{(\dot{\bm e}_{I}^u,\dot{\bm{e}}_{I}^w)}_{\rm E_p}\norm{(\bm e_{h}^u,\bm{e}_{h}^w)}_{\rm E_p} \, ds 
   +  \int_0^t \trinorm{\dot{\bm e}^\Sigma_I}_{{\rm E_f}}\norm{\bm e^\Sigma_h}_{{\rm E_f}}  \, ds \\ 
     + \int_0^t \mathcal{I}_h^f(\dot{\bm e}_I^\Sigma)\Big(\norm{ \dot{\bm{e}}_{h}^u}_{\Omega_p} + \norm{\dot{\bm{e}}_{h}^w}_{\Omega_p} \Big) \, ds 
       + \int_0^t \mathcal{I}_h^p(\ddot{\bm e}_I^u,\ddot{\bm e}_I^w) \norm{\bm{e}_h^\Sigma}_{{\rm E_f}} \, ds \\
   +  \int_0^t \mathcal{D}^p(\dot{\bm{e}}_{I}^w, \dot{\bm{e}}_{I}^w) \, ds  
+ \mathcal{A}^p_h((\bm{e}_{I}^u,\bm{e}_{I}^w),(\bm e_I^u,\bm e_I^w)) \\ + 
 \int_0^t 
    \Big(\mathcal{M}^f(\dot{\bm e}_{I}^\Sigma,\dot{\bm e}_I^\Sigma) + \mathcal{D}^f(\dot{\bm e}_{I}^\Sigma,\dot{\bm e}_I^\Sigma)\Big) \, ds 
    + \mathcal{A}_h^f(\bm{e}_{I}^\Sigma, \bm e_{I}^\Sigma)  + \mathcal{I}_h^p(\dot{\bm e}_I^u,\dot{\bm e}_I^w) \norm{\bm{e}_h^\Sigma}_{{\rm E_f}}. 
\end{multline}
We bound all terms by using the definition of the norms, except the last one for which we employ Young inequality for $\epsilon >0$ 
\begin{multline}
\circled{4}  \lesssim 
\int_0^t \Big( \trinorm{(\dot{\bm e}_{I}^u,\dot{\bm{e}}_{I}^w)}_{\rm E_p}
+ \trinorm{\dot{\bm e}^\Sigma_I}_{{\rm E_f}}
 + \mathcal{I}_h^f(\dot{\bm e}_I^\Sigma)
 + \mathcal{I}_h^p(\ddot{\bm e}_I^u,\ddot{\bm e}_I^w)
\Big)
\Big( \norm{(\bm e_{h}^u,\bm{e}_{h}^w)}_{\rm E_p} 
 + \norm{\bm e^\Sigma_h}_{{\rm E_f}} \Big) \, ds \\ 
   +  \int_0^t \mathcal{D}^p(\dot{\bm{e}}_{I}^w, \dot{\bm{e}}_{I}^w) \, ds  
+ \mathcal{A}^p_h((\bm{e}_{I}^u,\bm{e}_{I}^w),(\bm e_I^u,\bm e_I^w))  + 
 \int_0^t 
    \Big(\mathcal{M}^f(\dot{\bm e}_{I}^\Sigma,\dot{\bm e}_I^\Sigma) + \mathcal{D}^f(\dot{\bm e}_{I}^\Sigma,\dot{\bm e}_I^\Sigma)\Big) \, ds 
   \\ + \mathcal{A}_h^f(\bm{e}_{I}^\Sigma, \bm e_{I}^\Sigma)  + \frac{1}{2\epsilon}\mathcal{I}_h^p(\dot{\bm e}_I^u,\dot{\bm e}_I^w)^2 + \frac{\epsilon}{2}\norm{\bm e^\Sigma_h}_{{\rm E_f}}^2  
    \\
    \lesssim
\int_0^t \Big( \trinorm{(\dot{\bm e}_{I}^u,\dot{\bm{e}}_{I}^w)}_{\rm E_p}
+ \trinorm{\dot{\bm e}^\Sigma_I}_{{\rm E_f}}
 + \mathcal{I}_h^f(\dot{\bm e}_I^\Sigma)
 + \mathcal{I}_h^p(\ddot{\bm e}_I^u,\ddot{\bm e}_I^w)
\Big)
\Big( \norm{(\bm e_{h}^u,\bm{e}_{h}^w)}_{\rm E_p} 
 + \norm{\bm e^\Sigma_h}_{{\rm E_f}} \Big) \, ds \\ 
   +  \int_0^t \Big( \trinorm{(\dot{\bm e}_{I}^u,\dot{\bm{e}}_{I}^w)}_{\rm E_p}^2 + \trinorm{\dot{\bm e}^\Sigma_I}_{{\rm E_f}}^2 \Big)\, ds  
+ \trinorm{(\bm e_{I}^u,\bm{e}_{I}^w)}_{\rm E_p}^2 + \trinorm{\bm e^\Sigma_I}_{{\rm E_f}}^2 \\ + \frac{1}{2\epsilon}\mathcal{I}_h^p(\dot{\bm e}_I^u,\dot{\bm e}_I^w)^2 + \frac{\epsilon}{2}\norm{\bm e^\Sigma_h}_{{\rm E_f}}^2
= \circled{5}.
\end{multline}
{\color{black} Moreover, we observe that this bound for $\circled{4}$ is an upper bound also for $\circled{4}+\mathcal J_\mathcal{B}^t\lesssim \circled{5}+\int_0^t\|{\boldsymbol{e}}_I^r\|_{\Omega_f}^2\,ds$, for sufficiently small $\epsilon_5,\epsilon_6$.}
Finally, we consider $\epsilon$ small enough and take the supremum over $(0,t]$ to get 
\begin{multline*}
\sup_{t\in(0,T]}\| (\bm e_h^u, \bm e_h^w)(t) \|^2_{{\rm E}_p} + \| \bm e_h^\Sigma (t) \|^2_{{\rm E_f}}
    \lesssim
\int_0^T \mathcal{J}_1(s)
\Big( \norm{(\bm e_{h}^u,\bm{e}_{h}^w)}_{\rm E_p} 
 + \norm{\bm e^\Sigma_h}_{{\rm E_f}} \Big) \, ds \\ 
   +  \int_0^T \mathcal{J}_2(s) \, ds  
+ \sup_{t\in(0,T]}\mathcal{J}_3(t) {\color{black}+\int_0^T\|{\boldsymbol{e}}_I^r\|_{\Omega_f}^2\,ds},
\end{multline*}
where
\begin{align*}
\mathcal{J}_1 & =  \trinorm{(\dot{\bm e}_{I}^u,\dot{\bm{e}}_{I}^w)}_{\rm E_p}
+ \trinorm{\dot{\bm e}^\Sigma_I}_{{\rm E_f}}
 + \mathcal{I}_h^f(\dot{\bm e}_I^\Sigma)
 + \mathcal{I}_h^p(\ddot{\bm e}_I^u,\ddot{\bm e}_I^w) , \\
 \mathcal{J}_2 & =  \trinorm{(\dot{\bm e}_{I}^u,\dot{\bm{e}}_{I}^w)}_{\rm E_p}^2 + \trinorm{\dot{\bm e}^\Sigma_I}_{{\rm E_f}}^2, \\
 \mathcal{J}_3 &  = \trinorm{(\bm e_{I}^u,\bm{e}_{I}^w)}_{\rm E_p}^2 + \trinorm{\bm e^\Sigma_I}_{{\rm E_f}}^2  + \mathcal{I}_h^p(\dot{\bm e}_I^u,\dot{\bm e}_I^w)^2.
\end{align*}
 By applying the Gronwall Lemma we obtain
\begin{equation*}
\sup_{t\in(0,T]}\| (\bm e_h^u, \bm e_h^w, \bm e_h^\Sigma) (t) \|_{\rm E}
    \lesssim 
\int_0^T \mathcal{J}_1(s) \, ds + 
  \Big( \int_0^T \mathcal{J}_2(s) \, ds  
+ \sup_{t\in(0,T]}\mathcal{J}_3(t)
+\int_0^T\|{\boldsymbol{e}}_I^r\|_{\Omega_f}^2\,ds
\Big)^{\frac12}   
\end{equation*}
We conclude the proof by using the results in Lemma \ref{lem::interp2} and estimate the terms $\mathcal{I}_h^f(\cdot)$ and $\mathcal{I}_h^p(\cdot,\cdot)$ by using  \cite[Lemma 33]{CangianiDongGeorgoulisHouston_2017} as follows
\begin{gather*}
\mathcal{I}_h^f(\dot{\bm e}_I^\Sigma)^2
\lesssim
\sum_{\kappa\in\mathcal{T}^f_{h,I}}\frac{h_\kappa^{2q_\kappa-2}}{p_{f,\kappa}^{2n-3}}
\norm{\widetilde{\mathcal{E}}\dot{\bm \Sigma_f}}_{n,\mathcal{K}_\kappa}^2,
\\
\mathcal{I}_h^p(\dot{\bm e}_I^u, \dot{\bm e}_I^w)^2
\lesssim
\sum_{\kappa\in\mathcal{T}^p_{h,I}}\frac{h_\kappa^{2s_\kappa-2}}{p_{p,\kappa}^{2m-3}}
\norm{\widetilde{\mathcal{E}}\dot{\bm u}_p}_{m,\mathcal{K}_\kappa}^2	+
\sum_{\kappa\in\mathcal{T}^p_{h,I}}\frac{h_\kappa^{2r_\kappa-2}}{p_{p,\kappa}^{2\ell-3}}
\norm{\widetilde{\mathcal{E}}\dot{\bm w}_p}^2_{\ell,\mathcal{K}_\kappa}, \\
\mathcal{I}_h^p(\ddot{\bm e}_I^u, \ddot{\bm e}_I^w)^2
\lesssim
\sum_{\kappa\in\mathcal{T}^p_{h,I}}\frac{h_\kappa^{2s_\kappa-2}}{p_{p,\kappa}^{2m-3}}
\norm{\widetilde{\mathcal{E}}\ddot{\bm u}_p}_{m,\mathcal{K}_\kappa}^2	+
\sum_{\kappa\in\mathcal{T}^p_{h,I}}\frac{h_\kappa^{2r_\kappa-2}}{p_{p,\kappa}^{2\ell-3}}
\norm{\widetilde{\mathcal{E}}\ddot{\bm w}_p}^2_{\ell,\mathcal{K}_\kappa}
\end{gather*}
\end{proof}

\section{Time integration}\label{sec:time_int}

To integrate in time \eqref{eq:weak-dg}
we introduce in $\Omega_p \times (0,T]$ the auxiliary variables $\bm v_{ph} = \dot{\bm u}_{ph}$ and $\bm z_{ph} = \dot{\bm w}_{ph}$ and write the following (modified) formulation: for any $t \in (0,T]$ find 
$(\bm u_{ph}, \bm w_{ph}, \bm v_{ph}, \bm z_{ph}, \bm \Sigma_{fh}, \bm r_{fh})(t) \in  \bm W_h  = \bm V_h^p \times \bm V_h^p \times \bm V_h^p \times \bm V_h^p \times \bm S_h^f \times \bm \Lambda_h^f$ s.t.  
\begin{multline*}
(\dot{\bm u}_{ph}-\bm v_{ph} ,  \hat{\bm v})_{\Omega_p} + (\dot{\bm w}_{ph}-\bm z_{ph}, \hat{ \bm z})_{\Omega_p} + 
\mathcal{M}^p((\dot{\bm v}_{ph},\dot{\bm{z}}_{ph}),(\hat{\bm u},\hat{\bm w})) + 
 \mathcal{D}^p(\bm{z}_{ph}, \hat{\bm w})  \\+ \mathcal{A}^p_h((\bm{u}_{ph},\bm{w}_{ph}),(\hat{\bm{u}},\hat{\bm{w}})) 
 + \mathcal{M}^f(\dot{\bm{\Sigma}}_{fh},\hat{\bm{\tau}}) +  \mathcal{D}^f(\dot{\bm{\Sigma}}_{fh}, \hat{\bm \tau}) + \mathcal{A}_h^f(\bm{\Sigma}_{fh}, \hat{\bm{\tau}}) 
\\ + \mathcal{B}^f(\bm r_{fh}, \hat{\bm \tau}) - \mathcal{B}^f(\hat{\bm \lambda}, \dot{\bm \Sigma}_{fh}) 
+ \mathcal{C}^{pf} ((\dot{\bm{u}}_{ph},\dot{\bm{w}}_{ph}),\hat{\bm{\tau}})  - \mathcal{C}^{fp} (\dot{\bm{\Sigma}}_{fh},(\hat{\bm{u}},\hat{\bm{v}})) = \mathcal{F}(\hat{\bm{u}},\hat{\bm{w}},\hat{\bm{\tau}}) 
\end{multline*}
for any $(\hat{\bm{u}},\hat{\bm{w}}, \hat{\bm{v}},\hat{\bm{z}},\hat{\bm{\tau}}, \hat{\bm \lambda}) \in \bm W_h$, 
with initial conditions $\bm{u}_{ph}(0)=\bm u_{0h}, \bm{v}_{ph}(0)=\bm{v}_{0h}, \bm{w}_{ph}(0) = \bm{w}_{0,h}, \bm{z}_{ph}(0)=\bm{z}_{0h}$ and $\bm{\Sigma}_{fh}(0)=\bm{0}$.

By fixing a basis for the spaces $\bm V_h^p, \bm S_h^f$ and $\bm \Lambda_h^f$, and denoting by $\bm X(t) = (\bm U_p,\bm W_p, \bm V_p, \bm Z_p, \bm S_f, \bm R_f)^T \in \mathbb{R}^{ndof}$ the vector of the
$ndof$ expansion coefficients in the chosen basis, the above system can be written equivalently as
\begin{multline}
    \begin{bmatrix}
        I^p & 0 & 0 & 0 & 0 & 0 \\
        0 & I^p & 0 & 0 & 0 & 0 \\ 
        0 & 0 & M^p_{\rho} & M^p_{\rho_f} & -(N_{\alpha} + T)^T & 0 \\
        0 & 0 & M^p_{\rho_f} & M^p_{\rho_w} & -N^T & 0 \\
        0 & 0 & 0 & 0 & M^f + D^f_\delta & 0 \\
        0 & 0 & 0 & 0 & {B^f}^T & 0 
    \end{bmatrix}
    \begin{bmatrix}
        \dot{\bm U}_p \\ \dot{\bm W}_p \\ \dot{\bm V}_p \\ \dot{\bm Z}_p \\ \dot{\bm S}_f \\    \dot{\bm R}_f
    \end{bmatrix}
    \\ +     
    \begin{bmatrix}
        0 & 0 & -I^p & 0 & 0 & 0 \\
        0 & 0 & 0 & -I^p & 0 & 0 \\ 
        A^e + B^p_{\beta^2} & B^p_\beta  & 0 & 0 & 0 & 0 \\
        B^p_\beta & B^p & 0 & D^p_{\eta \kappa} + D^p_\gamma & 0 & 0 \\
        0 & 0 & N_\alpha + T & N & A^f & B^f \\
        0 & 0 & 0 & 0 & 0 & 0 
    \end{bmatrix}
    \begin{bmatrix}
    {\bm U}_p \\ {\bm W}_p \\ {\bm V}_p \\ {\bm Z}_p \\ {\bm S}_f \\ {\bm R}_f
    \end{bmatrix} = 
        \begin{bmatrix}
    {\bm 0} \\ {\bm 0} \\ {\bm F}_p \\ {\bm G}_p \\ {\bm H}_f \\ {\bm 0}
    \end{bmatrix},\label{eq:algebraic-system}
\end{multline}
with $\bm X(0) = \bm X_0 = (\bm U_0,\bm W_0, \bm V_0, \bm Z_0, \bm 0, \bm 0)^T$. 
In \eqref{eq:algebraic-system} the block  matrices
\begin{equation*}
    M^p = \begin{bmatrix}
        M^p_{\rho} & M^p_{\rho_f} \\
        M^p_{\rho_f} & M^p_{\rho_w}
    \end{bmatrix}
    \quad {\rm and} \quad 
    A^p = \begin{bmatrix}
        A^e + B^p_{\beta^2} & B^p_\beta \\
        B^p_\beta & B^p
    \end{bmatrix},
\end{equation*}
are the algebraic representation of the bilinear forms $\mathcal{M}^p(\cdot, \cdot)$ and $\mathcal{A}^p_h(\cdot,\cdot)$, respectively.
The damping matrix $D^p_{\eta \kappa} + D_\gamma^p, $ is associated with $\mathcal{D}^p(\cdot,\cdot)$, while
$M^f, D_\delta^f$, $A^f$ and $B^f$ to $\mathcal{M}^f(\cdot,\cdot), \mathcal{D}^f(\cdot,\cdot), \mathcal{A}_h^f(\cdot,\cdot)$ and $\mathcal{B}^f(\cdot,\cdot)$, respectively.
$N, N_\alpha$ and $T$ are related to the coupling terms in $\mathcal{C}^{fp}(\cdot, \cdot)$. 
Now, we rewrite problem \eqref{eq:algebraic-system} in a compact form as: \begin{equation}\label{eq:system-alg-compact}
    \begin{cases}
        {\rm M} \dot{\bm X}(t) + {\rm A} \bm X(t) = \bm F(t), &  t \in (0,T],\\
        \bm X(0) = \bm X_0, &
    \end{cases}
\end{equation}
and  partition the interval $[0, T]$ by introducing a time step $\Delta t > 0$ and define the following finite sequence of temporal steps
$t^k = k \Delta t$ for $k = 0, ...., N_T$, being $N_T = T /\Delta t$. Finally, we integrate system \eqref{eq:system-alg-compact} by using a $\theta$-method scheme with $\theta \in [1/2, 1]$, cf.~\cite{qss2007}, and get for $k=1,...,N_T$
\begin{equation}\label{eq:teta-metodo}
    ({\rm M} + \Delta t \theta{\rm A})\bm X^{k+1} =  ({\rm M} - \Delta t (1-\theta){\rm A})\bm X^{k} + \Delta t(\theta \bm F^{k+1} + (1-\theta) \bm F^{k}), 
\end{equation}
with $\bm X^k =  \bm{X}(t^k)$.

\section{Numerical results}\label{sec:nume_res}
The results obtained in this section have been achieved through the  \textsc{Matlab} code \texttt{lymph} \cite{antonietti2024lymphdiscontinuouspolytopalmethods}. The verification of the numerical scheme is presented in the 
first and second test for which we consider problems \eqref{eq:biot-system} and \eqref{eq:stokes-sigma} with the following modified coupling conditions on $\Gamma_I \times (0,T]$: 
\begin{equation}\label{eq:coupling-conditions-modified}
    \begin{cases}
        (\alpha \dot{\bm{u}}_p + \dot{\bm{w}}_p) \cdot \bm{n}_p  = \bm{u}_f \cdot \bm{n}_p - f_I^1, & \\
        \dot{\bm{\Sigma}}_f \bm{n}_p \cdot \bm{n}_p = 
        \gamma \dot{\bm{w}}_p \cdot \bm{n}_p - p_p + f_I^2, &  \\
        \alpha \dot{\bm{\Sigma}}_f \bm{n}_p \cdot \bm{n}_p
        = \bm{\sigma}_p \bm{n}_p \cdot \bm{n}_p - f_I^3, & \\
        \dot{\bm{\Sigma}}_f \bm{n}_p \wedge \bm{n}_p = \bm{\sigma}_p \bm{n}_p \wedge \bm{n}_p - f_I^4, & \\
        \dot{\bm{\Sigma}}_f \bm{n}_p \wedge \bm{n}_p = 
        \delta (\bm{u}_f - \dot{\bm{u}}_p) \wedge \bm{n}_p - f_I^5, & \\
    \end{cases}
\end{equation}
with $f_I^i$, for $i=1,...,5$ properly defined to obtain a reference  solution.
The last example concerns an application of geophysical interest.

\begin{figure}
\begin{minipage}{\textwidth}
\begin{minipage}{0.5\textwidth}
    \centering
    \includegraphics[width=1\linewidth]{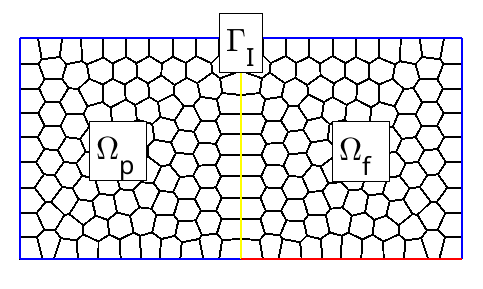}
    \caption{Test Case 1 and 2: polygonal mesh with Dirichlet (blue), Neumann (red), and interface (yellow) boundaries.}
    \label{fig:domain_convergence}
    \end{minipage}
 \hspace*{1cm}    
\begin{minipage}{0.35\textwidth}
\centering
\begin{tabular}{|c|c|c|}
\hline
\textbf{Field}      & \textbf{Test 1} & \textbf{Test 2} \\ \hline
$\rho_f$, $\rho_s$ & 1  & 1                         \\ \hline
$\lambda$, $\mu$   & 1,1 &   1, 0.5                        \\ \hline
$a$   & 1               & 1            \\ \hline
$\phi$   & 0.5           & 0.5                \\ \hline
$\eta/\kappa$             & 1   & 1                \\ \hline
$\rho_w$           & 2      & 2                 \\ \hline
$\beta$, $m$         & 1       & 1                \\ \hline
$\mu_f$ & 0.5 & 0.5 \\ \hline
$\alpha$ & 1 & 2                       \\
\hline
$\delta$ & 1 & 1                       \\
\hline
$\gamma$ & 0 & 0\\ \hline
\end{tabular}
\caption{Test case 1 and 2.  Physical parameters.}
\label{param}
\end{minipage}
\end{minipage}    
\end{figure}

\subsection{Test case 1}
We consider $\Omega = (-1,1) \times (0,1)$ with $\Omega_p = (-1,0) \times (0,1)$ and  $\Omega_f = (0,1) \times (0,1)$ such that $\Gamma_I = \{0 \}\times (0,1)$, cf. Figure \ref{fig:domain_convergence}.  We fix the final time $T=0.1$, $\Delta t = 0.001$ and chose $\theta=\frac12$ in \eqref{eq:teta-metodo} (Crank-Nicolson scheme). We select the interface parameters $\alpha=\delta =1$, and $\gamma=0$, set the 
analytical solution as
\begin{equation}\label{eq:test1_analytical_sol}
    \bm u_p = \begin{bmatrix}
     x \frac{t^2}{4} \\ t x^2 \frac{y}{2} - y \frac{t^3}{6}
    \end{bmatrix},
    \; 
    \bm w_p = \begin{bmatrix}
       - t x \frac{y^2}{2} + x \frac{t^3}{6} \\ y \frac{t^2}{4} 
    \end{bmatrix},
    \; \bm \Sigma_f = \begin{bmatrix}
    \frac{t^3}{6} - t^2 \frac{(y^2-x^2)}{4} & 0 \\ 0 & - \frac{t^3}{6} - t^2 \frac{(y^2-x^2)}{4}
    \end{bmatrix},
\end{equation}
and compute the remaining data accordingly.
In particular, in \eqref{eq:coupling-conditions-modified} we have $f_I^1 = f_I^2 = f_I^4 = f^I_5 = 0$ while $f_I^3 = -\frac{1}{6}t^3 + \frac{3}{4}t^2$. The physical parameters considered are listed in Figure \ref{param}. In Figure \ref{fig:testcase1_convergence} (left), resp. (right), we report the computed error $\norm{(\bm e^u,\bm e^w)}_{\rm E_p}$, resp.  $\norm{ \bm e^\Sigma}_{\rm E_f}$, as a function of the mesh size $h_p$, resp. $h_f$, by choosing a polynomial degree equal to 1 and 2. The results agree with the theoretical estimates shown in Theorem \ref{thm::error-estimate}.
In Figure \ref{fig:testcase1_convergence_p} we plot the computed errors $\norm{(\bm e^u,\bm e^w)}_{\rm E_p}$ and $\norm{ \bm e^\Sigma}_{\rm E_f}$ with respect to the polynomial degree $p_p = p_f = 1,...5$ for different choices of the the time step: $\Delta t = 0.001$ (left) and $\Delta t = 0.0001$ (right). As expected, since the analytical solution is polynomial, cf. \eqref{eq:test1_analytical_sol}, the error curves decay exponentially until the threshold $\mathcal{O}(\Delta t^2)$, given by the time integration scheme \eqref{eq:teta-metodo}, is reached.

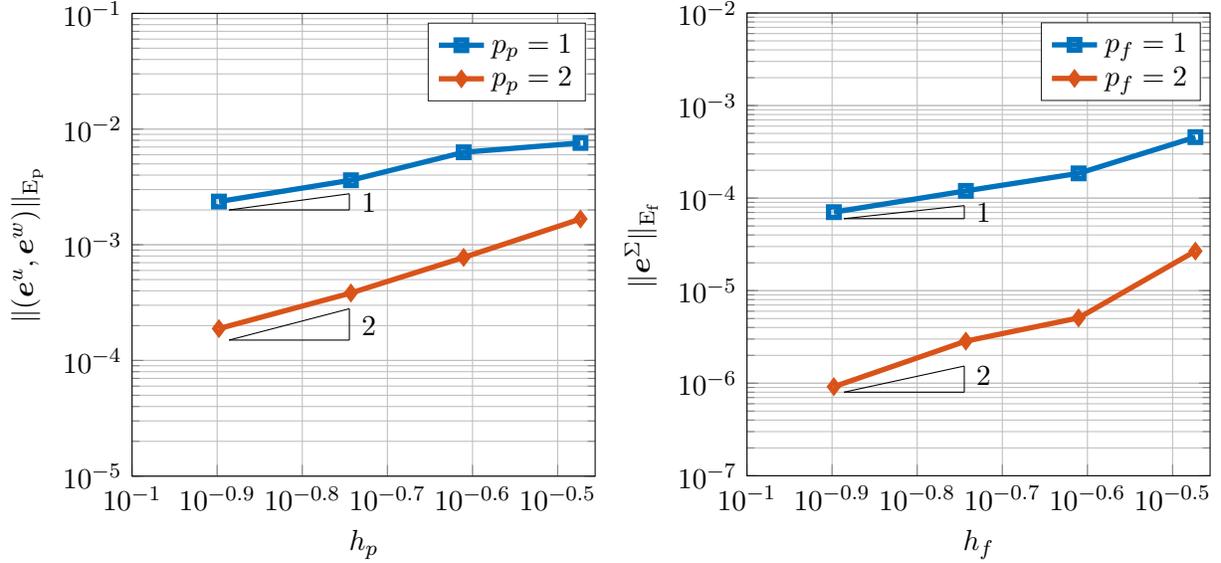
\begin{figure}
    \centering
%
%
\definecolor{mycolor1}{rgb}{0.00000,0.44700,0.74100}%
\definecolor{mycolor2}{rgb}{0.85000,0.32500,0.09800}%
\definecolor{mycolor3}{rgb}{0.92900,0.69400,0.12500}%
\begin{tikzpicture}

\begin{axis}[%
width=0.35\textwidth,
height=0.35\textwidth,
at={(0.78in,0.521in)},
scale only axis,
xmode=log,
xmin=0.1,
xmax=0.35,
xlabel style={font=\color{black}},
xlabel={$h_p$},
xminorticks=true,
ymode=log,
ymin=1e-05,
ymax=0.1,
yminorticks=true,
ylabel style={font=\color{black}},
ylabel={$\norm{(\bm e^u, \bm e^w)}_{\rm E_p}$},
axis background/.style={fill=white},
xmajorgrids,
xminorgrids,
ymajorgrids,
yminorgrids,
legend style={legend cell align=left, align=left, draw=white!15!black}
]

\addplot [color=mycolor1, line width=2.0pt, mark=square, mark options={solid, mycolor1}]
  table[row sep=crcr]{%
0.336355009488299	7.591887589746592e-03\\
0.245201900765537	6.302001906734867e-03\\
0.180773859749347	3.612722981843438e-03\\
0.126486924607499	2.363823568826487e-03\\
};
\addlegendentry{$p_p=1$}

\addplot [color=mycolor2, line width=2.0pt,mark=diamond, mark options={solid, mycolor2}]
  table[row sep=crcr]{%
0.336355009488299	1.670529602307839e-03\\
0.245201900765537	7.768656972746108e-04\\
0.180773859749347	3.831395354815264e-04\\
0.126486924607499	1.884747579919900e-04\\
};
\addlegendentry{$p_p=2$}

\addplot [color=black, forget plot]
  table[row sep=crcr]{%
0.13	2.e-03\\
0.18	2.76e-03\\
};

\addplot [color=black, forget plot]
  table[row sep=crcr]{%
0.18	2.e-03\\
0.18	2.76e-03\\
};
\addplot [color=black, forget plot]
  table[row sep=crcr]{%
0.13	2.e-03\\
0.18	2.e-03\\
};

\node[right, align=left, text=black, font=\normalsize]
at (axis cs:0.181,2.3e-3) {1};

\addplot [color=black, forget plot]
  table[row sep=crcr]{%
0.13	1.5e-04\\
0.18	2.8e-04\\
};

\addplot [color=black, forget plot]
  table[row sep=crcr]{%
0.18	1.5e-04\\
0.18	2.8e-04\\
};
\addplot [color=black, forget plot]
  table[row sep=crcr]{%
0.13	1.5e-04\\
0.18	1.5e-04\\
};

\node[right, align=left, text=black, font=\normalsize]
at (axis cs:0.181,2.e-4) {2};

\end{axis}
\end{tikzpicture}%
%
%
\definecolor{mycolor1}{rgb}{0.00000,0.44700,0.74100}%
\definecolor{mycolor2}{rgb}{0.85000,0.32500,0.09800}%
\definecolor{mycolor3}{rgb}{0.92900,0.69400,0.12500}%
\begin{tikzpicture}

\begin{axis}[%
width=0.35\textwidth,
height=0.35\textwidth,
at={(0.78in,0.521in)},
scale only axis,
xmode=log,
xmin=0.1,
xmax=0.35,
xminorticks=true,
xlabel style={font=\color{black}},
xlabel={$h_f$},
ymode=log,
ymin=1e-07,
ymax=0.01,
yminorticks=true,
axis background/.style={fill=white},
xmajorgrids,
xminorgrids,
ymajorgrids,
yminorgrids,
ylabel style={font=\color{black}},
ylabel={$\norm{\bm e^\Sigma}_{\rm E_f}$},
legend style={legend cell align=left, align=left, draw=white!15!black}
]

\addplot [color=mycolor1, line width=2.0pt, mark=square, mark options={solid, mycolor1}]
  table[row sep=crcr]{%
0.336355009488299     4.559216955591511e-04\\
0.245201900765537     1.851609617315513e-04\\
0.180773859749347     1.197463557603900e-04\\
0.126486924607499     7.054989781683092e-05\\
};
\addlegendentry{$p_f=1$}

\addplot [color=mycolor2, line width=2.0pt, mark=diamond, mark options={solid, mycolor2}]
  table[row sep=crcr]{%
0.336355009488299	2.677377345533187e-05\\
0.245201900765537	5.073554481429559e-06\\
0.180773859749347	2.849855015942675e-06\\
0.126486924607499	9.172592107167155e-07\\
};
\addlegendentry{$p_f=2$}

\addplot [color=black, forget plot]
  table[row sep=crcr]{%
0.13	6.e-05\\
0.18	8.3e-05\\
};

\addplot [color=black, forget plot]
  table[row sep=crcr]{%
0.18	6.e-05\\
0.18	8.3e-05\\
};
\addplot [color=black, forget plot]
  table[row sep=crcr]{%
0.13	6.e-05\\
0.18	6.e-05\\
};

\node[right, align=left, text=black, font=\normalsize]
at (axis cs:0.181,7.e-5) {1};

\addplot [color=black, forget plot]
  table[row sep=crcr]{%
0.13	8.e-07\\
0.18	1.533e-06\\
};

\addplot [color=black, forget plot]
  table[row sep=crcr]{%
0.18	8.e-07\\
0.18	1.533e-06\\
};
\addplot [color=black, forget plot]
  table[row sep=crcr]{%
0.13	8.e-07\\
0.18	8.e-07\\
};

\node[right, align=left, text=black, font=\normalsize]
at (axis cs:0.181,1.2e-6) {2};

\end{axis}
\end{tikzpicture}%
  \caption{Test Case 1. Left: log-log plot of the computed error $\norm{(\bm e^u, \bm e^w)}_{\rm E_p}$ as a function of the mesh size $h_p$ for $p_p = 1,2$. Right: log-log plot of the computed error
$\norm{\bm e^\Sigma}_{\rm E_f}$ as a function of the mesh size $h_f$ for $p_f = 1,2$. Final time $T=0.1$ and $\Delta t = 0.001$. }
    \label{fig:testcase1_convergence}
\end{figure}

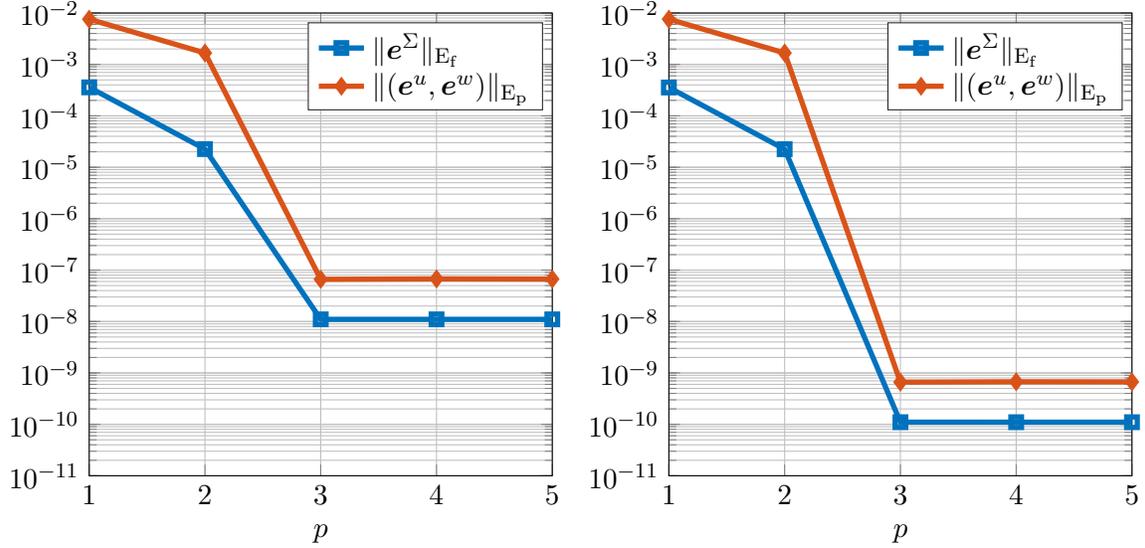
\begin{figure}
    \centering
%
%
\definecolor{mycolor1}{rgb}{0.00000,0.44700,0.74100}%
\definecolor{mycolor2}{rgb}{0.85000,0.32500,0.09800}%
\begin{tikzpicture}

\begin{axis}[%
width=0.35\textwidth,
height=0.35\textwidth,
at={(0.78in,0.521in)},
scale only axis,
xmin=1,
xmax=5,
xlabel={$p$},
ymode=log,
ymin=1e-11,
ymax=0.01,
yminorticks=true,
axis background/.style={fill=white},
xmajorgrids,
ymajorgrids,
yminorgrids,
legend style={legend cell align=left, align=left, draw=white!15!black}
]
\addplot [color=mycolor1, line width=2.0pt, mark=square, mark options={solid, mycolor1}]
  table[row sep=crcr]{%
1	0.000358455335373836\\
2	2.25654917023717e-05\\
3	1.10208593080805e-08\\
4	1.10142968184293e-08\\
5	1.10129600298562e-08\\
};
\addlegendentry{$\norm{\bm e^\Sigma}_{\rm E_f}$}

\addplot [color=mycolor2, line width=2.0pt, mark=diamond, mark options={solid, mycolor2}]
  table[row sep=crcr]{%
1	0.00759188758974659\\
2	0.00167052960230784\\
3	6.6189100259971e-08\\
4	6.74381612220654e-08\\
5	6.69473629227672e-08\\
};
\addlegendentry{$\norm{(\bm e^u, \bm e^w)}_{\rm E_p}$}

\end{axis}
\end{tikzpicture}%
%
%
\definecolor{mycolor1}{rgb}{0.00000,0.44700,0.74100}%
\definecolor{mycolor2}{rgb}{0.85000,0.32500,0.09800}%
\begin{tikzpicture}

\begin{axis}[%
width=0.35\textwidth,
height=0.35\textwidth,
at={(0.78in,0.521in)},
scale only axis,
xmin=1,
xmax=5,
xlabel={$p$},
ymode=log,
ymin=1e-11,
ymax=0.01,
yminorticks=true,
axis background/.style={fill=white},
xmajorgrids,
ymajorgrids,
yminorgrids,
legend style={legend cell align=left, align=left, draw=white!15!black}
]
\addplot [color=mycolor1, line width=2.0pt, mark=square, mark options={solid, mycolor1}]
  table[row sep=crcr]{%
1	0.000356081386333717\\
2	2.24591266397099e-05\\
3	1.10165721943702e-10\\
4	1.10102563629454e-10\\
5	1.10097624777985e-10\\
};
\addlegendentry{$\norm{\bm e^\Sigma}_{\rm E_f}$}

\addplot [color=mycolor2, line width=2.0pt, mark=diamond, mark options={solid, mycolor2}]
  table[row sep=crcr]{%
1	0.00759235229431997\\
2	0.00166991831022002\\
3	6.61661368134932e-10\\
4	6.73085688898942e-10\\
5	6.70093174348001e-10\\
};
\addlegendentry{$\norm{(\bm e^u, \bm e^w)}_{\rm E_p}$}

\end{axis}
\end{tikzpicture}%
  \caption{Test Case 1. Semi-log plot of the computed errors $\norm{(\bm e^u, \bm e^w)}_{\rm E_p}$ and $\norm{\bm e^\Sigma}_{\rm E_f}$ as a function of the polynomial degree $p=p_p = p_f$ fixing the number of mesh element equal to $100$. Final time $T=0.1$ and time step $\Delta t = 0.001$ (left), $\Delta t = 0.0001$ (right). }
    \label{fig:testcase1_convergence_p}
\end{figure}

\subsection{Test case 2}
With the same setup of the previous test case and using the parameters in Table \ref{param}, we consider the following analytical solution:
\begin{equation*}
    \bm u_p = e^{-t} \begin{bmatrix}
     \sin(x-y) \\ \sin(x-y) 
    \end{bmatrix},
    \; 
    \bm w_p = - \bm u_p,
    \; \bm \Sigma_f = (e^{-t}-1)  \begin{bmatrix}
    \cos(x-y) & 0 \\ 0 & - \cos(x-y)
    \end{bmatrix},
\end{equation*}
where $\bm \Sigma^f$ is obtained by selecting 
\begin{equation*}
    \bm u_f = -e^{-t} \begin{bmatrix}
     \sin(x-y) \\ \sin(x-y) 
    \end{bmatrix},
    \; {\rm and} \;
    p_f = 0 \; {\rm in } \; \Omega_f.  
\end{equation*}
The remaining data are computed accordingly, and in particular, we set $f_I^1 = f_I^4 = f_I^5 = 0$ while $f_I^2 = -e^{-t} \cos(y)$ and $f_I^3 = 3e^{-t}\cos(y)$.
We report in Figure \ref{fig:testcase2_convergence} (left) the computed error $\norm{(\bm e^u, \bm e^w, \bm e^\Sigma)}_{\rm E}$, as a function of the mesh size $h = \max(h_p,h_f)$, for a polynomial degree $p=p_p=p_f$ ranging from 1 to 4. The results agree with the theoretical results of Theorem \ref{thm::error-estimate}.
In Figure \ref{fig:testcase2_convergence} the computed error $\norm{(\bm e^u, \bm e^w, \bm e^\Sigma)}_{\rm E}$ is shown as a function of the polynomial degree $p=p_p = p_f = 1,...5$ fixing the number of mesh element equal to 100. Also in this case the numerical results are aligned with the theoretical estimates in Theorem \ref{thm::error-estimate}.

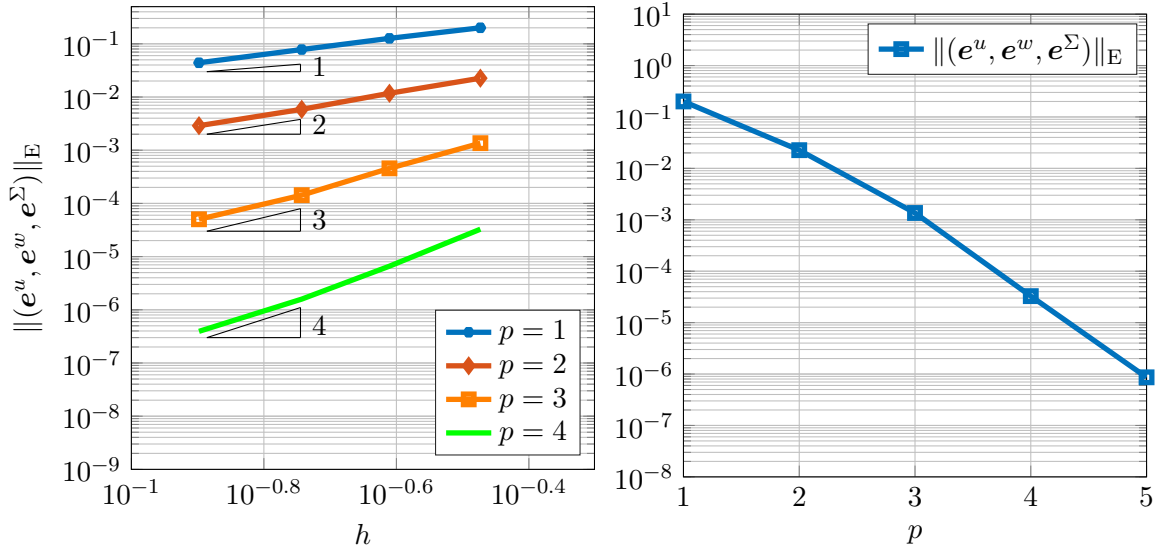
\begin{figure}
    \centering
%
%
\definecolor{mycolor1}{rgb}{0.00000,0.44700,0.74100}%
\definecolor{mycolor2}{rgb}{0.85000,0.32500,0.09800}%
\definecolor{mycolor3}{rgb}{0.92900,0.69400,0.12500}%
\begin{tikzpicture}

\begin{axis}[%
width=0.35\textwidth,
height=0.35\textwidth,
at={(0.78in,0.521in)},
scale only axis,
xmode=log,
xmin=0.1,
xmax=0.5,
xlabel style={font=\color{black}},
xlabel={$h$},
xminorticks=true,
ymode=log,
ymin=1e-09,
ymax=0.5,
yminorticks=true,
ylabel style={font=\color{black}},
ylabel={$\norm{(\bm e^u, \bm e^w, \bm e^\Sigma)}_{\rm E}$},
axis background/.style={fill=white},
xmajorgrids,
xminorgrids,
ymajorgrids,
yminorgrids,
legend style={legend cell align=left, align=left, draw=white!15!black},
legend pos=south east
]

\addplot [color=mycolor1, line width=2.0pt, mark=asterisk, mark options={solid, mycolor1}]
  table[row sep=crcr]{%
0.336355009488299	0.200385383178282\\
0.245201900765537	0.126678127662751\\
0.180773859749347	0.0780792041867877\\
0.126486924607499	0.0440455800319966\\
};
\addlegendentry{$p=1$}

\addplot [color=mycolor2, line width=2.0pt,mark=diamond, mark options={solid, mycolor2}]
  table[row sep=crcr]{%
0.336355009488299	0.0225977020118877\\
0.245201900765537	0.0117550778280885\\
0.180773859749347	0.00587143901658651\\
0.126486924607499	0.00289099806032648\\
};
\addlegendentry{$p=2$}

\addplot [color=orange, line width=2.0pt, mark=square, mark options={solid, orange}]
  table[row sep=crcr]{%
0.336355009488299	0.00135668598601958\\
0.245201900765537	0.000456274914108075\\
0.180773859749347	0.00014286591083415\\
0.126486924607499	5.03407224263044e-05\\
};
\addlegendentry{$p=3$}

\addplot [color=green, line width=2.0pt, mark=triangle-, mark options={solid, green}]
  table[row sep=crcr]{%
0.336355009488299	3.27502567648369e-05\\
0.245201900765537	6.60928144590893e-06\\
0.180773859749347	1.58814787890381e-06\\
0.126486924607499	3.93462671739523e-07\\
};
\addlegendentry{$p=4$}

\addplot [color=black, forget plot]
  table[row sep=crcr]{%
0.13	3.e-02\\
0.18	0.0415\\
};

\addplot [color=black, forget plot]
  table[row sep=crcr]{%
0.18	3.e-02\\
0.18	0.0415\\
};
\addplot [color=black, forget plot]
  table[row sep=crcr]{%
0.13	3.e-02\\
0.18	3.e-02\\
};

\node[right, align=left, text=black, font=\normalsize]
at (axis cs:0.181,3.8e-2) {1};

\addplot [color=black, forget plot]
  table[row sep=crcr]{%
0.13	2e-03\\
0.18	0.0038\\
};

\addplot [color=black, forget plot]
  table[row sep=crcr]{%
0.18	2e-03\\
0.18	0.0038\\
};
\addplot [color=black, forget plot]
  table[row sep=crcr]{%
0.13	2e-03\\
0.18	2e-03\\
};

\node[right, align=left, text=black, font=\normalsize]
at (axis cs:0.181,3.e-3) {2};

\addplot [color=black, forget plot]
  table[row sep=crcr]{%
0.13	3e-05\\
0.18	7.9636e-05\\
};

\addplot [color=black, forget plot]
  table[row sep=crcr]{%
0.18	3e-05\\
0.18	7.9636e-05\\
};
\addplot [color=black, forget plot]
  table[row sep=crcr]{%
0.13	3e-05\\
0.18	3e-05\\
};

\node[right, align=left, text=black, font=\normalsize]
at (axis cs:0.181,5.e-5) {3};

\addplot [color=black, forget plot]
  table[row sep=crcr]{%
0.13	3e-07\\
0.18	1.1027e-06\\
};

\addplot [color=black, forget plot]
  table[row sep=crcr]{%
0.18	3e-07\\
0.18	1.1027e-06\\
};
\addplot [color=black, forget plot]
  table[row sep=crcr]{%
0.13	3e-07\\
0.18	3e-07\\
};

\node[right, align=left, text=black, font=\normalsize]
at (axis cs:0.181,5.e-7) {4};

\end{axis}
\end{tikzpicture}%
%
%
\definecolor{mycolor1}{rgb}{0.00000,0.44700,0.74100}%
\definecolor{mycolor2}{rgb}{0.85000,0.32500,0.09800}%
\begin{tikzpicture}

\begin{axis}[%
width=0.35\textwidth,
height=0.35\textwidth,
at={(0.78in,0.521in)},
scale only axis,
xmin=1,
xmax=5,
xlabel={$p$},
ymode=log,
ymin=1e-8,
ymax=10,
yminorticks=true,
axis background/.style={fill=white},
xmajorgrids,
ymajorgrids,
yminorgrids,
legend style={legend cell align=left, align=left, draw=white!15!black}
]
\addplot [color=mycolor1, line width=2.0pt, mark=square, mark options={solid, mycolor1}]
  table[row sep=crcr]{%
1	0.200385383178282\\
2	0.0225977020118877\\
3	0.00135668598601065\\
4	3.27502568228775e-05\\
5	8.6339051248388e-07\\
};
\addlegendentry{$\norm{(\bm e^u, \bm e^w,\bm e^\Sigma)}_{\rm E}$}

\end{axis}
\end{tikzpicture}%
  \caption{Test Case 2. Left: log-log plot of the computed error $\norm{(\bm e^u, \bm e^w, \bm e^\Sigma)}_{\rm E}$ as a function of the mesh size $h = \max(h_p,h_f)$ for $p = p_p = p_f = 1,2,3,4$. Right: semi-log plot of the computed error
$\norm{(\bm e^u, \bm e^w, \bm e^\Sigma)}_{\rm E}$ as a function of the polynomial degree $p=p_p = p_f$ fixing the number of mesh element equal to $100$. Final time $T=0.1$ and time step $\Delta t = 0.001$. }
    \label{fig:testcase2_convergence}
\end{figure}

\subsection{Test case 3}

In this last example, we apply our method to a problem similar to the one presented in \cite{LiYotov2022} which is motivated by the coupling of surface and subsurface hydrological systems. On the domain $\Omega = (0, 2) \times (-1, 1)$, 
we associate the upper half to $\Omega_f$ and the lower half to $\Omega_p$. This can be interpreted as a surface flow (lake or river) modeled by the Stokes problem over a poroelastic aquifer, governed by the Biot system.
In each subdomain, we consider 800 polygonal elements, see Figure \ref{fig:domain_phys}, and polynomial degrees $p_p = p_f = 3$ for a final simulation time $T=1.5~s$ and time step $\Delta t=0.01~s$. The appropriate interface conditions are enforced along the interface $\Gamma_I = \{y = 0\}$.
We consider two cases with different values of $ \eta/\kappa$, $m$, $\lambda_p$ and $\delta$, as described in Figure \ref{param_phys}.
\begin{figure}
\begin{minipage}{\textwidth}
\begin{minipage}{0.65\textwidth}
    \centering
    \includegraphics[width=0.8\linewidth]{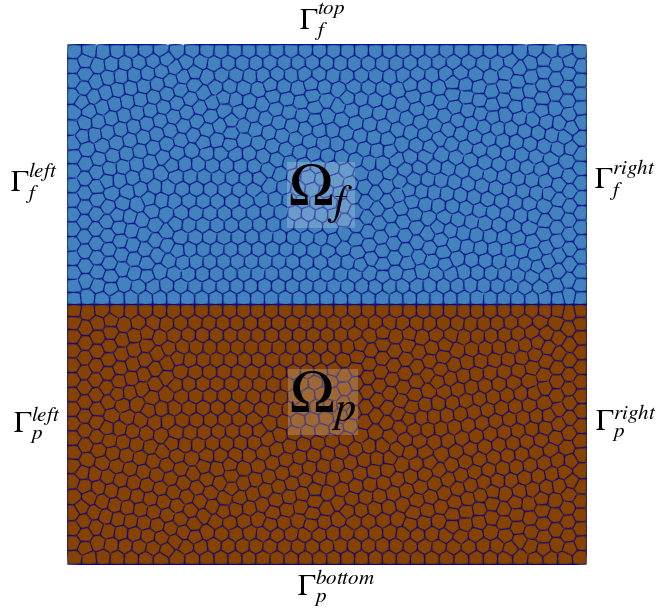}
    \caption{Test Case 3. Fluid $\Omega_f = (0,2)\times(0,1)$ and poroelastic $\Omega_p = (0,2)\times (-1,0)$ domains. Polygonal mesh with 1600 elements.}
    \label{fig:domain_phys}
    \end{minipage}
\hfill
\begin{minipage}{0.3\textwidth}
\centering
\begin{tabular}{|c|c|c|}
\hline
\textbf{Field}      & \textbf{Set A} & \textbf{Set B} \\ \hline
$\rho_f$, $\rho_s$ & 1  & 1                         \\ \hline
$\lambda$  & 1 &   $10^6$                        \\ \hline
$\mu$   & 1 &   1                        \\ \hline
$a$   & 1               & 1            \\ \hline
$\phi$   & 0.5           & 0.5                \\ \hline
$\eta/\kappa$             & 1   & $10^4$                \\ \hline
$\rho_w$           & 2      & 2                 \\ \hline
$\beta$         & 1      & 1        \\ \hline        $m$         & 1      & $10^4$                \\ \hline
$\mu_f$ & 0.5 & 0.5 \\ \hline
$\alpha$ & 1 & 1                       \\
\hline
$\delta$ & 1 & 100                       \\
\hline
$\gamma$ & 1 & 1\\ \hline
\end{tabular}
\caption{Test case 3.  Physical parameters.}
\label{param_phys}
\end{minipage}
\end{minipage}    
\end{figure}

The body forces and external source are zero, as well as the initial conditions. The flow is driven by a parabolic fluid velocity imposed on the left boundary of the fluid region. The corresponding boundary conditions are as follows:
\begin{equation*}
    \begin{cases}
        \nabla \cdot \bm \Sigma_f = h(t)(-40y(y-1),0)^T& on \; \Gamma_f^{left} \times (0,T],\\
        \nabla \cdot \bm \Sigma_f = \bm 0 & on \; \Gamma_f^{top} \times (0,T], \\
        \Sigma_f \bm n_f = \bm 0 & on \; \Gamma_f^{right} \times (0,T],\\
        p_p  = 0 & on \; \Gamma_p^{bottom} \times (0,T],\\
       \bm \sigma_p \bm n_p = \bm 0 & on \; \Gamma_p^{bottom} \times (0,T],\\
         \bm u_p = \bm 0 & on \; \Gamma_p^{left} \cup \Gamma_p^{right} \times (0,T],\\
         \bm w_p\cdot \bm n_p =  0 & on \; \Gamma_p^{left} \cup \Gamma_p^{right} \times (0,T],\\         
    \end{cases}
\end{equation*}
where $h(t) = 1/(1 + e^{-10(t-1)})$. 
For each case, we present the plots of computed velocities and pressure 
at final time $T$.
In particular, in $\Omega_f$ we compute $\bm u_f$ using \eqref{eq:1st-eq-rewrite} and $p_f = -\frac12 {\rm tr}(\dot{\bm \Sigma_f})$, while in $\Omega_p$ we use \eqref{eq:constitutive-poro} to obtain $p_p$ while $\dot{\bm u}_p$ and $\dot{\bm w}_p$ are directly inferred from \eqref{eq:teta-metodo}.
From the velocity plots, cf. Figures \ref{fig:caseA} and \ref{fig:caseB} (left), we observe that the fluid is driven into the poroelastic medium due to zero pressure
at the bottom, which simulates gravity. 

\begin{figure}
    \centering
    \includegraphics[width=0.45\linewidth]{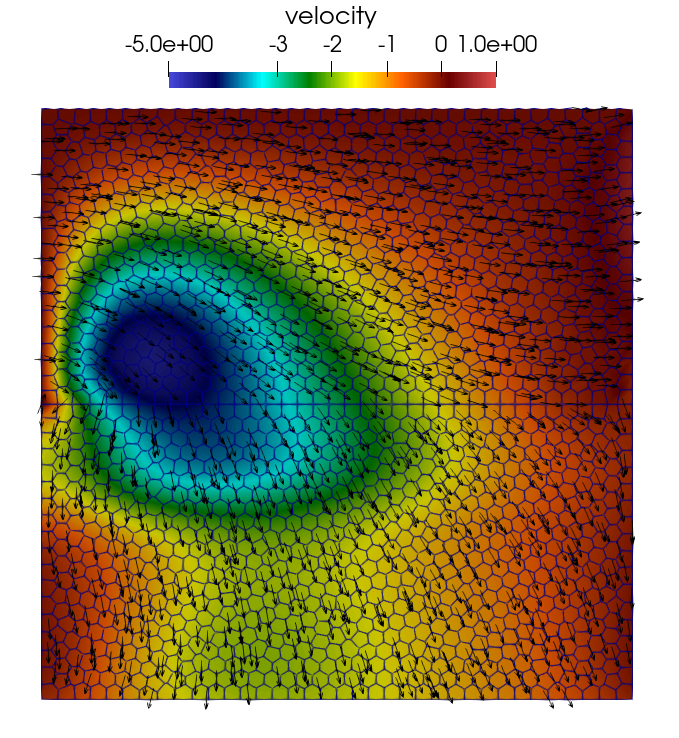}
    \includegraphics[width=0.45\linewidth]{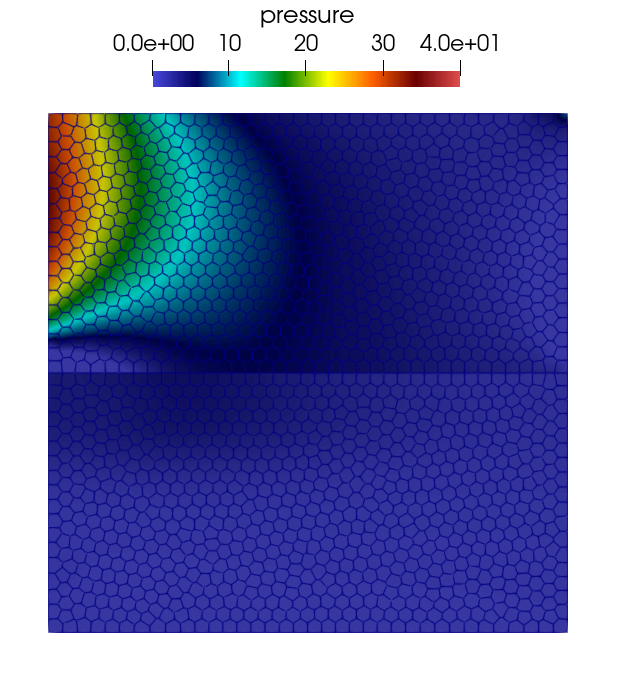}    
    \caption{Test case 3: set A. Computed
solutions at final time $T = 1.5~s$. Left: velocities $\bm u_f$ and $\dot{\bm u}_p + \dot{\bm w}_p$ (arrows), $\bm u_{f,2}$ and $\dot{\bm u}_{p,2} + \dot{\bm w}_{p,2}$
(color). Right: computed pressures $p_f$ and $p_p$.}
    \label{fig:caseA}
\end{figure}

The mass conservation $(\alpha \dot{\bm{u}}_p + \dot{\bm{w}}_p) \cdot \bm{n}_p  = \bm{u}_f \cdot \bm{n}_p$ on the interface with $\bm n_p = (0, 1)^T$ indicates continuity of second components of these two velocity vectors, which is observed from the color plot of the velocity, see cf. Figures \ref{fig:caseA} and \ref{fig:caseB} (left). 
We observe large values for the fluid pressure near the left boundary, which is due to the inflow
condition. A discontinuity close to the left lower corner $(0,0)$ appears due to the mismatch in inflow boundary conditions between the fluid and poroelastic regions.
These results are in agreement with \cite{LiYotov2022}.

\begin{figure}
    \centering
    \includegraphics[width=0.45\linewidth]{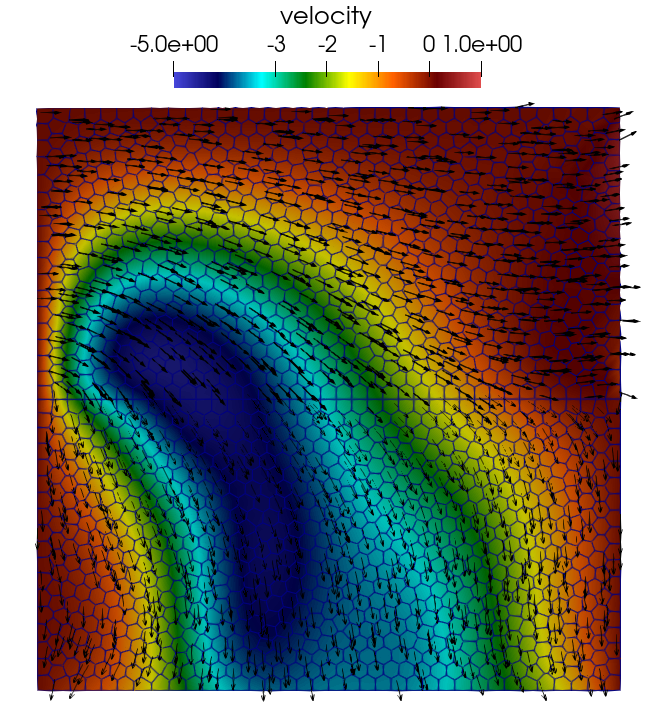}
    \includegraphics[width=0.45\linewidth]{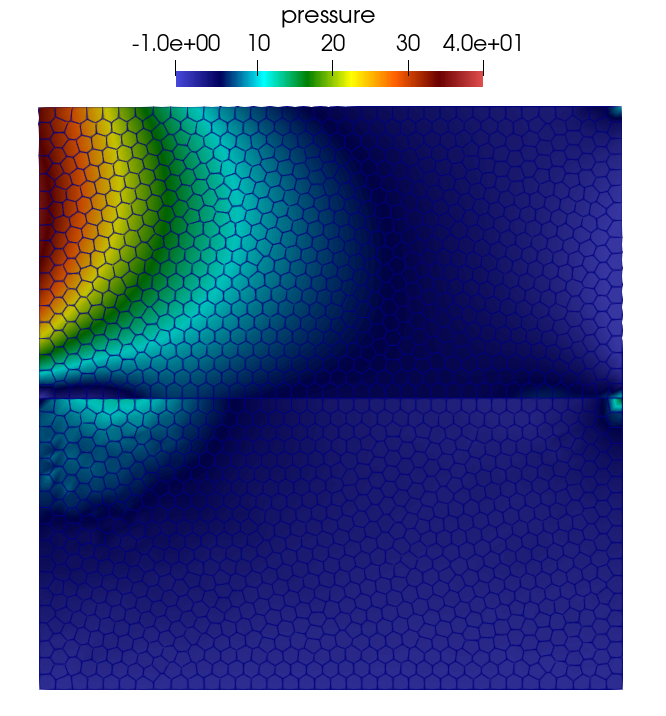}    
    \caption{Test case 3: set B. Computed
solutions at final time $T = 1.5~s$. Left: velocities $\bm u_f$ and $\dot{\bm u}_p + \dot{\bm w}_p$ (arrows), $\bm u_{f,2}$ and $\dot{\bm u}_{p,2} + \dot{\bm w}_{p,2}$
(color). Right: computed pressures $p_f$ and $p_p$.}
    \label{fig:caseB}
\end{figure}

For the set B,  the model problem exhibits both locking regimes for poroelasticity: (i) small
permeability and storativity and (ii) almost incompressible material as observed in \cite{SONYOUNG2017}. In particular, the Poisson’s ratio $\nu = \frac{\lambda_p}{2(\lambda_p + \nu_p)} = 0.4999995$. The computed solution does not exhibit locking or oscillations. The behavior is qualitatively similar to set A, with larger fluid and 
poroelastic pressure, see Figure ~\ref{fig:caseB} (right).

\section{Conclusions}\label{sec:conc}
This study has presented a comprehensive numerical analysis of a polygonal discontinuous Galerkin scheme for simulating fluid exchange between a deformable, saturated poroelastic structure and an adjacent free-flow channel. The investigation specifically addressed wave phenomena governed by the low-frequency Biot model in the poroelastic region and unsteady Stokes flow in the free-flow domain. Transmission conditions enforce conservation laws, achieving coupling at the interface between the two regions and ensuring robust interaction between the subsystems. Spatial discretization relied on the two-displacement weak form of the poroelasticity system and a stress-based formulation of the Stokes equation with weakly imposed symmetry. A thorough stability analysis of the proposed semi-discrete formulation was conducted, confirming the robustness of the method. Furthermore, a-priori $hp$-error estimates were derived, providing theoretical guarantees on the accuracy and convergence of the numerical scheme. These findings establish a solid foundation for the reliable and efficient simulation of coupled poroelastic and fluid-flow systems using advanced DG methods as has been shown in the numerical examples.



\section*{Acknowledgements}
IF and IM have been partially supported by ICSC—Centro Nazionale di Ricerca in High Performance Computing, Big Data, and Quantum Computing funded by European Union—NextGenerationEU. 
The present research is part of the activities of ``Dipartimento di Eccelllenza 2023-2027". The authors are members of INdAM-GNCS.

\section*{Declarations}

\textbf{Conflict of interest/Competing interests} 
The authors have no conflicts of interest to declare that are relevant to the content of this article. 
\textbf{Data availability}. The datasets generated during the current study are available from Ilario Mazzieri on reasonable request.

\begin{appendices}

\section{Proof of inf-sup inequality \eqref{eq:infsup}}\label{sec:inf-sup}


We first observe that \eqref{eq:infsup} is equivalent to be able to find, for each $\boldsymbol{\lambda}_h\in\boldsymbol{\Lambda}_h^f$ a $\boldsymbol{\tau}_h\in\boldsymbol{S}_h^f$ such that
\begin{equation}\label{eq:rewrite-infsup}
\mathcal B^f(\boldsymbol{\lambda}_h,\boldsymbol{\tau}_h)=\|\boldsymbol{\lambda}_h\|_{\Omega_f}^2
\qquad\text{ and }\qquad
\|\boldsymbol{\tau}_h\|_{\star}\lesssim \|\boldsymbol{\lambda}_h\|_{\Omega_f},
\end{equation}
where
\[
\|\boldsymbol{\tau}\|_\star \overset{\rm def}{=}
\|\boldsymbol{\tau}\|_{\text{E}_f}+\displaystyle\sum_{\kappa\in\mathcal{T}_{h,p}^I}p_{p,\kappa}h_\kappa^{-1/2}\|\boldsymbol{\tau}\|_{\partial\kappa\cap \Gamma_I}.
\]
We thus construct such a $\boldsymbol{\tau}_h$ by extending the analysis of \cite{boffi2009reduced}, to include the interface terms 
of the norm $\|\boldsymbol{\tau}_h\|_\star$.
This construction is carried on considering
\begin{itemize}
    \item the two-dimensional case $\Omega_f\subset \mathbb R^2$;
    \item the fact that our discontinuous space ${\bm S}_h^f\times\boldsymbol{\Lambda}_h^f$ includes the Amara-Thomas space \cite{amara1979equilibrium}.
\end{itemize}
The extension to the three-dimensional case is not trivial: as indicated in \cite{boffi2009reduced}, a more complex or completely alternative approach should be considered, and also a different auxiliary finite element space.

\bigskip

This proof relies on the following property, which is
verified if we take ${\boldsymbol{\Psi}}={\bm H}^1_{0,\Gamma_I}(\Omega)$, ${\boldsymbol{\Psi}}_h$ is one of the finite element spaces considered in \cite{boffi2009reduced}, and $\mathfrak{S}_h^f$ is the tensor space associated to it:
\begin{equation}\label{eq:hp-infsup}
\forall \boldsymbol{\tau}\in{\bm S}^f\quad \exists\boldsymbol{\tau}_h^1\in{\mathfrak S}_h^f\text{ s.t. }
\begin{cases}
    (\nabla\cdot(\boldsymbol{\tau}-\boldsymbol{\tau}_h^1),{\bm v}_h)_{\Omega_f} = 0\quad \forall{\bm v}_h\in{\boldsymbol{\Psi}}_h\subset{\boldsymbol{\Psi}},\\
    \|\boldsymbol{\tau}_h^1\|_{\bm S^f}\lesssim\|\boldsymbol{\tau}\|_{\bm S^f},
\end{cases}
\end{equation}
where $\|\boldsymbol{\tau}\|_{\bm S^f}^2=\|(2\mu_f)^{-1/2}\text{dev}(\boldsymbol{\tau})\|_{\Omega_f}^2+\|\rho_f^{-1/2}\nabla\cdot\boldsymbol{\tau}\|_{\Omega_f}^2$.

Following the proof of \cite[Proposition 2]{boffi2009reduced}, we introduce a discrete space $\mathfrak{L}_h^f$ approximating $\boldsymbol{\Lambda}^f$ and for a fixed $\boldsymbol{\lambda}_h\in\mathfrak{L}_h^f$ we can build up a continuous tensor $\boldsymbol{\tau}\in{\bm S}^f$ such that
\begin{equation}\label{eq:cont-rewrite-infsup}
\mathcal B^f(\boldsymbol{\lambda}_h,\boldsymbol{\tau})=\|\boldsymbol{\lambda}_h\|_{\Omega_f}^2
\qquad\text{ and }\qquad
\|\boldsymbol{\tau}\|_{\bm S^f}\lesssim \|\boldsymbol{\lambda}_h\|_{\Omega_f}.
\end{equation}
This tensor is defined as $\overline{\boldsymbol{\tau}} = \zeta(\boldsymbol{\psi}):=\begin{bmatrix}
    -\partial_2\psi_1 & \partial_1\psi_1\\-\partial_2\psi_2 & \partial_1\psi_2
\end{bmatrix}$, where $\boldsymbol{\psi}\in {\boldsymbol{\Psi}}$ is the velocity component of the solution to the following Stokes problem:
\[\begin{cases}
\mathfrak{a}(\boldsymbol{\psi},\boldsymbol{\varphi}) + (p,\nabla\cdot\boldsymbol{\varphi})_{\Omega_f}=0
& \forall\boldsymbol{\varphi}\in{\boldsymbol{\Psi}},\\
(q,\nabla\cdot\boldsymbol{\psi})_{\Omega_f} = (\mathfrak{s}(q),\boldsymbol{\lambda}_h)_{\Omega_f} & \forall q\in Q=L^2(\Omega_f),
\end{cases}\]
where
\[
\mathfrak{s}(q)=\begin{bmatrix}
    0 & q \\ -q & 0
\end{bmatrix},
\quad
\mathfrak{a}(\boldsymbol{\psi},\boldsymbol{\varphi}) = (2\mu_f\varepsilon(\boldsymbol{\psi}),\varepsilon(\boldsymbol{\varphi}))_{\Omega_f}. 
\]
Indeed, $\mathfrak{a}$ is coercive over ${\boldsymbol{\Psi}}$, $\mathcal B^f(\boldsymbol{\lambda}_h,\overline{\boldsymbol{\tau}})=(\mathfrak{s}(\nabla\cdot\boldsymbol{\psi}), \boldsymbol{\lambda}_h)=\|\nabla\cdot\boldsymbol{\psi}\|_{\Omega_f}^2=\|\boldsymbol{\lambda}_h\|_{\Omega_f}^2$, and the continuity inequality in \eqref{eq:cont-rewrite-infsup} follows from 
$\|\overline{\boldsymbol{\tau}}\|_{\bm S^f}^2\lesssim\mathfrak{a}(\boldsymbol{\psi},\boldsymbol{\psi})$
and classical Stokes analysis.

Now, we construct a projection $\Pi_h:{\bm S}^f\to{\mathfrak{S}}_h^f$ such that $\overline{\boldsymbol{\tau}}_h=\Pi_h\overline{\boldsymbol{\tau}}$ satisfies \eqref{eq:rewrite-infsup}.
Taking a fixed $\boldsymbol{\tau}\in{\bm S}^f$, we define $\Pi_h\boldsymbol{\tau}=\boldsymbol{\tau}_h^1+\boldsymbol{\tau}_h^2$ as a sum of two terms.
The first one is the $\boldsymbol{\tau}_h^1$ corresponding to \eqref{eq:hp-infsup}.
The second one is defined as $\boldsymbol{\tau}_h^2=\zeta(\boldsymbol{\psi}_h)$, where $\boldsymbol{\psi}_h$ is the solution of the following discrete Stokes problem over an inf-sup stable pair of discrete spaces ${\boldsymbol{\Psi}}_h\times Q_h\subset {\boldsymbol{\Psi}}\times Q$:
\begin{equation}\label{eq:stokes-infsup}\begin{cases}
\mathfrak{a}(\boldsymbol{\psi}_h,\boldsymbol{\varphi}_h) + (p_h,\nabla\cdot\boldsymbol{\varphi}_h)_{\Omega_f}= 0
& \forall\boldsymbol{\varphi}_h\in{\boldsymbol{\Psi}}_h,\\
(q_h,\nabla\cdot\boldsymbol{\psi}_h)_{\Omega_f} =\mathcal B^f(\mathfrak{s}(q_h),\overline{\boldsymbol{\tau}}-\boldsymbol{\tau}_h^1) & \forall q_h\in Q_h.
\end{cases}\end{equation}

Since $\mathfrak{s}$ is a bijection between $Q_h$ and $\mathfrak{L}_h^f$,
we can denote by $\overline{q}_h$ the element of $Q_h$ such that $\mathfrak{s}(\overline{q}_h)=\boldsymbol{\lambda}_h$ and observe that $\mathcal B^f(\boldsymbol{\lambda}_h,\boldsymbol{\tau}_h^2) = (\mathfrak{s}(\overline{q}_h), \nabla\cdot\boldsymbol{\psi}_h) = \mathcal{B}^f(\boldsymbol{\lambda}_h,\overline{\boldsymbol{\tau}}-\boldsymbol{\tau}_h^1)$.
Moreover, classical Stokes analysis yields $\|\boldsymbol{\tau}_h^2\|_{\bm S^f}\lesssim\|\overline{\boldsymbol{\tau}}-\boldsymbol{\tau}_h^1\|_{\bm S^f}$.
Summarizing, we end up with a $\overline{\boldsymbol{\tau}}_h=\boldsymbol{\tau}_h^1+\boldsymbol{\tau}_h^2$ that satisfies 
\[
\mathcal{B}^f(\boldsymbol{\lambda}_h,\overline{\boldsymbol{\tau}}_h)=\|\boldsymbol{\lambda}_h\|_{\Omega_f}^2 \qquad \text{and}\qquad \|\overline{\boldsymbol{\tau}}_h\|_{\bm S^f}\lesssim\|\boldsymbol{\lambda}_h\|_{\Omega_f}.
\]
Now, $\overline{\boldsymbol{\tau}}_h$ is continuous over $\Omega_f$ by construction and $\overline{\boldsymbol{\tau}}_h|_{\Gamma_I}=\boldsymbol{0}$ because of the zero Dirichlet condition encoded in the spaces $\boldsymbol{\Psi}, \boldsymbol{\Psi}_h$ to which $\boldsymbol{\psi},\boldsymbol{\psi}_h$ belong, respectively.
Therefore $\|\overline{\boldsymbol{\tau}}_h\|_{\bm S^f}=\|\overline{\boldsymbol{\tau}}_h\|_{\rm E_f}$.
Finally, observing that the dG spaces considered in this work are such that $\bm S_h^f\supset\mathfrak{S}_h^f$ and $\boldsymbol{\Lambda}_h=\mathfrak{L}_h^f$, the proof is complete.
\\
\qed




\end{appendices}


\end{document}